\documentclass[a4paper,11pt]{article}
\usepackage[T2A]{fontenc}
\usepackage[utf8]{inputenc}
\usepackage[russian,english]{babel}

\usepackage{hyperref}
\usepackage[noadjust]{cite}

\usepackage{geometry}
\usepackage{graphicx}
\usepackage{longtable}
\usepackage{multirow}
\usepackage{amsfonts}
\usepackage{amsthm}
\usepackage{amsmath}
\usepackage{amssymb}
\usepackage{accents}
\usepackage{indentfirst}
\usepackage{enumerate}
\usepackage{float}
\usepackage{placeins}
\usepackage{mathtools}
\usepackage{csquotes}
\usepackage{setspace}
\usepackage{xifthen}
\usepackage{pdfpages}
\usepackage{transparent}
\usepackage{cite}
\usepackage{authblk}
\usepackage{arydshln}

\usepackage{rotating}
\usepackage{pdflscape}
\usepackage{multicol}

\usepackage{enumitem}
\setlist[enumerate]{font=\upshape,label=(\arabic*)}

\newtheorem{theorem}{Theorem}[section]
\newtheorem{corollary}[theorem]{Corollary}
\newtheorem{proposition}[theorem]{Proposition}
\newtheorem{lemma}[theorem]{Lemma}
\theoremstyle{definition}
\newtheorem{definition}[theorem]{Definition}
\newtheorem{descript}[theorem]{Description}

\newtheorem{notation}[theorem]{Notation}
\newtheorem{example}[theorem]{Example}
\newtheorem{remark}[theorem]{Remark}

\numberwithin{equation}{section}
\allowdisplaybreaks

\let\phi\varphi

\renewcommand{\Re}{\mathop{\mathfrak{Re}}\nolimits}
\renewcommand{\Im}{\mathop{\mathfrak{Im}}\nolimits}

\def\A{\mathcal{A}}
\def\E{\mathcal{E}}
\def\Hyp{\mathcal{H}}
\def\Main{\mathcal{M}}

\DeclareMathOperator{\sgn}{sgn}

\DeclareMathOperator{\Lin}{Lin}
\DeclareMathOperator{\Anc}{Anc}
\DeclareMathOperator{\Ann}{Ann}

\newcommand\til[1]{\widetilde{\phantom{,\mkern-4mu} #1 \phantom{,\mkern-4mu}}\mkern-1mu}

\newcommand\numeq[1]%
  {\stackrel{\scriptstyle(\mkern-1.5mu#1\mkern-1.5mu)}{=}}

\providecommand{\keywords}[1]{\textbf{Keywords:} #1}
\providecommand{\msc}[1]{\textbf{MSC 2020:} #1}

\newcommand\freefootnote[1]{%
    \bgroup
    \renewcommand\thefootnote{\fnsymbol{footnote}}%
    \renewcommand\thempfootnote{\fnsymbol{mpfootnote}}%
    \footnotetext[0]{#1}%
    \egroup
}

\begin{document}

\title{Orthogonality graphs of real Cayley--Dickson algebras. Part I: Doubly alternative zero divisors and their hexagons}
\author{
Svetlana Zhilina$^{a,b}$
}
\date{\small \em
$^a$Department of Mathematics and Mechanics, Lomonosov Moscow State\\ University, Moscow, 119991, Russia\\
$^b$Moscow Center of Fundamental and Applied Mathematics, Moscow, 119991, Russia
}

\maketitle

\begin{abstract}
We study zero divisors whose components alternate strongly pairwise and construct oriented hexagons in the zero divisor graph of an arbitrary real Cayley--Dickson algebra. In case of the algebras of the main sequence, the zero divisor graph coincides with the orthogonality graph, and any hexagon can be extended to a double hexagon. We determine the multiplication table of the vertices of a double hexagon. Then we find a sufficient condition for three elements to generate an alternative subalgebra of an arbitrary Cayley--Dickson algebra. Finally, we consider those zero divisors whose components are both standard basis elements up to sign. We classify them and determine necessary and sufficient conditions under which two such elements are orthogonal.
\end{abstract}


\msc{05C25, 17A20, 17D05, 17D99.}

\freefootnote{This work was supported by the Russian Science Foundation (project No. 17-11-01124).}

\freefootnote{Email address: \texttt{s.a.zhilina@gmail.com}}

\section{Introduction}

The Cayley--Dickson algebras have been a subject of intensive research to mathematicians for decades. We mention in particular classical works by Schafer~\cite{schafer}, Eakin and Sathaye~\cite{eakin}, and a fundamental book by McCrimmon~\cite{mccrimmon}.

One of the most peculiar properties of Cayley--Dickson algebras $\A_n$ is that they stop being alternative for $n \geq 4$. As a consequence, there appear zero divisors which are hard to study and to classify, except for some particular cases. At present most of the authors restrict their attention to the algebras of the main sequence which we denote by~$\Main_n$. In these algebras all parameters which determine the Cayley--Dickson construction are assumed to be equal to $-1$. The most successful efforts have been taken by Moreno~\cite{moreno, moreno_constructing} and Biss, Dugger, and Isaksen~\cite{biss, biss2}. Moreno's key idea was to study doubly alternative zero divisors, that is, such elements that their components are both alternative elements of the previous algebra. This notion was extended in~\cite{our_split-algebras} to the Cayley--Dickson split-algebras, denoted by~$\Hyp_n$, and led to similar results.

We are also concerned with a convenient representation of relations between zero divisors. One of the possible solutions is to draw a zero divisor graph or an orthogonality graph of the algebra. Actually, studying relation graphs of various algebras has recently become a rapidly extending branch of mathematics.

Some previous works on relation graphs of real Cayley--Dickson algebras include~\cite{our_split-algebras, our_split-sedenions} where orthogonality, commutativity, and zero divisor graphs of low-dimensional ($n \leq 4$) Cayley--Dickson split-algebras have been described. 

The current work is devoted to studying zero divisors in arbitrary real Cayley--Dickson algebras whose components satisfy some additional conditions on the norm and alternativity. We are interested in the patterns which they form in orthogonality and zero divisor graphs, and these patterns appear to be hexagonal, see Corollary~\ref{corollary:A_n-double-hexagon}. We pay a particular attention to the algebras of the main sequence, where hexagons can be extended to the so-called double hexagons, see Description~\ref{description:double-hexagon}. The vertices of a double hexagon have a convenient multiplication table which has a block structure, cf. Theorem~\ref{theorem:neat-basis}. These three statements are the main results of our paper.

In Section~\ref{section:alternative-subalgebras}, based on the work by Moreno~\cite{moreno_alternative} on the algebras of the main sequence, we determine a sufficient condition for two or three elements to generate an associative or an alternative subalgebra in an arbitrary real Cayley--Dickson algebra. We achieve this by constructing a homomorphism from $\A_2$ or $\A_3$ to the subalgebra discussed. However, in contrast to the algebras of the main sequence, this homomorphism may have a nontivial kernel.

In Section~\ref{section:basis-pairs}
we study zero divisors whose components are both standard basis elements up to sign, that is, which are of the form $(e_i, \pm e_j)$, and determine the conditions under which two such elements are orthogonal. The key statement for their classification is Lemma~\ref{lemma:basis-pairs-zero-divisors-relation} which is similar to the following property of the sedenions: if $(a,b)(c,d) = 0$ in $\mathbb{S}$, then $n(c)ab = -n(a)cd$, cf.~\cite[Proposition~3.4(ii)]{chan}. In our next article we will describe an orthogonality graph on the vertices of the form $(\pm e_i, \pm e_j)$ for an arbitrary Cayley--Dickson algebra.

In some higher-dimensional ($n = 4,5,6$) algebras of the main sequence, zero divisors of the form $(e_i, \pm e_j)$ have been studied by de Marrais~\cite{marrais}. In case of the sedenions, he has also obtained an analogue of our double hexagons~\cite[p.~3]{marrais2} and the multiplication table of their vertices~\cite[p.~8]{marrais3}. We emphasize, however, that our proof is more general.

\section{An overview of real Cayley--Dickson algebras} \label{section:A_n}

\subsection{Algebraic relations and their graphs} \label{subsection:definitions}

Let $\mathbb{F}$ be an arbitrary field and $(\A, +, \cdot)$ be an algebra with a unity $1_{\A}$ over the field $\mathbb{F}$. $\A$ is not assumed to be commutative or associative. We say that $a$ and $b$ in $\A$ {\em anticommute} if $ab + ba =0 $, and $a$ and $b$ {\em are orthogonal} if $ab = ba = 0$. We denote the set of zero divisors (left, right, or two-sided) of $\A$ by $Z(\A)$, and the set of two-sided zero divisors of $\A$ by $Z_{LR}(\A)$.

\begin{definition}
Let $a$ be an arbitrary element of $\A$.
\begin{itemize}
    \item
    {\em The anticentralizer} of $a$ is $\Anc_\A(a) = \big\{ b \in \A \: | \:  ab+ba=0 \big\}$, namely, the set of all elements in $\A$ which anticommute with $a$.
    \item
    {\em The orthogonalizer} of $a$ is $O_\A(a)=\big\{ b \in \A \: | \;  ab=ba=0 \big\}$, namely, the set of all elements in $\A$ which are orthogonal to $a$.
    \item
    {\em The left annihilator} of $a$ is the set $l.\Ann_\A(a) = \big\{ b \in \A \: | \; ba=0 \big\}$.
    \item
    Similarly, {\em the right annihilator} of $a$ is $r.\Ann_\A(a) = \big\{ b \in \A \: | \; ab=0 \big\}$.
\end{itemize}
\end{definition}

Clearly, these subsets are vector spaces over $\mathbb{F}$.

\begin{notation}
For any set $X \subseteq \A$ we denote the set of lines passing through elements of $X$ by
$$
P(X) = \{ [x] = \mathbb{F} x \; | \; x \in X \}.
$$
\end{notation}

\begin{definition}
Let $\A$ be an arbitrary algebra.
\begin{itemize}
\item 
{\em The orthogonality graph} $\Gamma_O(\A)$ is defined as follows: its vertices are lines in $Z_{LR}(\A)$, that is,
$$
V(\Gamma_O(\A)) = P(Z_{LR}(\A)),
$$
and distinct vertices $[a]$ and $[b]$ are adjacent if and only if $ab=ba=0$.
\item
{\em The directed zero divisor graph} $\Gamma_Z(\A)$ is defined as follows: its vertices are lines in $Z(\A)$, that is,
$$
V(\Gamma_Z(\A)) = P(Z(\A)),
$$
and distinct vertices $[a]$ and $[b]$ form a directed edge $([a], [b])$ if and only if $ab=0$.
\end{itemize}
\end{definition}

Note that the edges of $\Gamma_O(\A)$ and $\Gamma_Z(\A)$ are well-defined. When speaking of the vertices of these graphs, we will not distinguish between a nonzero element $a$ and a line $[a] = \mathbb{F} a$ passing through it.

In an undirected graph $\Gamma$, $d(x,y) = d_{\Gamma}(x,y)$ denotes the distance between two vertices $x$ and $y$, and $d(\Gamma) = \sup\limits_{x,y \in \Gamma} d(x,y)$ denotes the diameter of $\Gamma$.

\subsection{Constructing Cayley--Dickson algebras} \label{subsection:A_n}

We refer the reader to~\cite{mccrimmon,schafer} for auxiliary definitions and general properties of Cayley--Dickson algebras.

\begin{definition} \label{definition:Cayley--Dickson-algebras}
Let $\A$ be an algebra over a field $\mathbb{F}$ with an involution $a \mapsto \bar{a}$. The algebra $\A \{ \gamma \}$ produced by the Cayley--Dickson process, when applied to $\A$ with the parameter $\gamma \in \mathbb{F}$, $\gamma \neq 0$, is defined as the set of ordered pairs of elements of $\A$ with operations
\begin{align*}
\alpha(a,b)&=(\alpha a, \alpha b);\\
(a,b)+(c,d)&=(a+c,b+d);\\
(a,b)(c,d)&=(ac+\gamma \bar{d}b,da+b\bar{c})
\end{align*} 
and the involution
$$
\qquad (\overline{a,b})=(\bar{a},-b), \qquad a,b,c,d\in \A, \ \alpha \in \mathbb{F}.
$$
If the involution on $\A$ is regular, that is, $a + \bar{a} \in \mathbb{F}1_{\A}$ and $a\bar{a} = \bar{a}a \in \mathbb{F}1_{\A}$ for all $a \in \A$, then the involution on $\A \{ \gamma \}$ is also regular, cf.~\cite[p.~435]{schafer}.
\end{definition}

Henceforth we assume that $\mathbb{F} = \mathbb{R}$. We now define an arbitrary real Cayley--Dickson algebra which is determined by the set of its parameters. The most common definition of real Cayley--Dickson algebras assumes that all parameters are equal to $-1$. These particular algebras are what we call the algebras of the main sequence.

\begin{definition} \label{definition:A_n}
For every integer $n \geq 0$ and nonzero real numbers $\gamma_0, \dots, \gamma_{n-1}$ we define the real Cayley--Dickson algebra $\A_n = \A_n \{ \gamma_0, \dots, \gamma_{n-1} \}$ inductively:
\begin{enumerate}
\item $\A_0 = \mathbb{R}$, and $e^{(0)}_0 = 1$ is its only basis element;
\item If $\A_n \{ \gamma_0, \dots, \gamma_{n-1} \}$ is constructed, then $\A_{n+1} \{ \gamma_0, \dots, \gamma_n \} =
\\
(\A_n \{ \gamma_0, \dots, \gamma_{n-1} \}) \{ \gamma_n \}$. Its basis elements are $e^{(n+1)}_0, \dots, e^{(n+1)}_{2^{n+1}-1}$ such that
\begin{equation*}
e^{(n+1)}_j =
\begin{cases}
(e^{(n)}_j,0), & 0 \leq j \leq 2^n-1,\\
(0,e^{(n)}_{j-2^n}), & 2^n \leq j \leq 2^{n+1}-1.
\end{cases}
\end{equation*}
\end{enumerate}
\end{definition}

For every integer $n \geq 0$ the structure $\A_n$ in Definition~\ref{definition:A_n} is a $2^n$-dimensional algebra over $\mathbb{R}$ with the unit element $e^{(n)}_0$ and a regular involution, cf.~\cite[Lemma 3.14]{our_split-algebras}.
Consider the following definitions which are analogous to those for complex numbers.

\begin{definition} \label{definition:real-imaginary-part}
\leavevmode
\begin{itemize}
	\item Let $a \in \A_n$. Its {\em real part} is $\Re(a) = \frac{a + \bar{a}}{2}$, its {\em imaginary part} is $\Im(a) = \frac{a - \bar{a}}{2}$, and its {\em norm} is $n(a) = a \bar{a} = \bar{a}a$. Since the involution on $\A_n$ is regular, $\Re(a), n(a) \in \mathbb{R}$.
	\item An element $a \in \A_n$ is said to be {\em pure} if $\Re(a) = 0$.
	\item An element $(a, b) \in \A_{n+1}$ is said to be {\em doubly pure} if $\Re(a) = \Re(b) = 0$.
	\item An element $(a, b) \in \A_{n+1}$ is {\em four times pure} if $a$ and $b$ are doubly pure.
\end{itemize}
\end{definition}

\begin{proposition}[{\cite[p. 435]{schafer}}]
We can compute real part and norm of an element $(a,b) \in \A_{n+1}$ inductively by using the following equalities:
\begin{align*}
	\Re((a,b)) &= \Re(a),\\
	n((a,b)) &= n(a) - \gamma_n n(b).
\end{align*}
\end{proposition}

\begin{notation}
We denote
\begin{align*}
    \E_n &= \left\{ e^{(n)}_0, e^{(n)}_1, \dots, e^{(n)}_{2^{n}-1} \right\},\\
    \E'_n &= \left\{ e^{(n)}_1, \dots, e^{(n)}_{2^{n}-1} \right\} = \E_n \setminus \left\{ e^{(n)}_0 \right\},\\
    \E''_n &= \left\{ e^{(n)}_1, \dots, e^{(n)}_{2^{n-1}-1}, e^{(n)}_{2^{n-1}+1}, \dots e^{(n)}_{2^{n}-1} \right\} = \E_n \setminus \left\{ e^{(n)}_0, e^{(n)}_{2^{n-1}} \right\}.
\end{align*}
Clearly, $\E'_n$ is the set of pure basis elements, and $\E''_n$ is the set of doubly pure basis elements. We also denote
$$
\A'_n = \{ a \in \A_n \; | \; \Re(a) = 0 \}.
$$
\end{notation}

\subsection{Some properties of real Cayley--Dickson algebras} \label{subsection:A_n-properties}

Henceforth we assume that $\A$ is an arbitrary algebra over a field $\mathbb{F}$, and $\A_n = \A_n \{ \gamma_0, \dots, \gamma_{n-1} \}$ is an arbitrary real Cayley--Dickson algebra. By~\cite[Exercise 2.5.1]{mccrimmon}, $\A_n \{ \gamma_0, \dots, \gamma_{n-1} \}$ is isomorphic to $\A_n \{ \sgn(\gamma_0), \dots, \sgn(\gamma_{n-1}) \}$, so it is sufficient to consider only $\gamma_k \in \{ \pm 1 \}$, $k = 0, \dots, n-1$.

\begin{proposition} \label{proposition:lambda-form}
Let $\langle a, b \rangle$ denote a real-valued symmetric bilinear form associated with the quadratic form $n(a)$. Then $\langle a, a \rangle = n(a)$ and $2\langle a,b \rangle = a \bar{b} + b \bar{a} = \bar{a} b + \bar{b} a$ for all $a, b \in \A_n$, cf. \cite[Proposition~3.18, Proposition~3.19]{our_split-algebras}.

We remark that for all $a, b \in \A_n$ we have $\langle a, b \rangle = \langle \bar{a}, \bar{b} \rangle$ and $\Re(a) = \langle a, e_0 \rangle$.
\end{proposition}

\begin{lemma}[{\cite[Lemma 1.3]{moreno}, \cite[Lemma 4.2]{our_split-algebras}}] \label{lemma:inner-product-movement}
Let $a, b, c \in \A_n$. Then $\langle a, bc \rangle = \langle a\bar{c}, b \rangle = \langle \bar{b}a, c \rangle$.
\end{lemma}

We now define two important sequences of Cayley--Dickson algebras.

\begin{definition}
\leavevmode
\begin{itemize}
	\item It is said that the algebra $\A_n \{ \gamma_0, \dots, \gamma_{n-1} \}$ is an algebra of {\em the main sequence} if $\gamma_k = -1$ for each $k = 0, \dots, n-1$. We denote this algebra by~$\Main_n$.
	\item The algebra $\A_n \{ \gamma_0, \dots, \gamma_{n-1} \}$ is called {\em a Cayley--Dickson split-algebra} if $\gamma_k = -1$ for each $k = 0, \dots, n-2$ and $\gamma_{n-1} = 1$.  We denote it by $\Hyp_n$, since the norm in $\Hyp_n$ appears to be hyperbolic.
\end{itemize}
\end{definition}

Clearly, $\Main_n$ and $\Hyp_n$ differ by the last parameter only. In other words, $\Main_n = \Main_{n-1} \{ -1 \}$ and $\Hyp_n = \Main_{n-1} \{ 1 \}$.

\begin{proposition}[{\cite[Proposition~3.31]{our_split-algebras}}] \label{proposition:A_n-euclidean-product}
\leavevmode
\begin{itemize}
	\item Let $a = \sum\limits_{j=0}^{2^n-1} a_j e^{(n)}_j, \ \ b = \sum\limits_{j=0}^{2^n-1} b_j e^{(n)}_j \in \Main_n$. Then $\langle a, b \rangle = \sum\limits_{j=0}^{2^n-1} a_j b_j$ is a Euclidean inner product. Particularly, $n(a) = \sum\limits_{j=0}^{2^n-1} a_j^2$, so $n(a)=0$ if and only if $a=0$.
	\item Let $a = \sum\limits_{j=0}^{2^n-1} a_j e^{(n)}_j, \ \ b = \sum\limits_{j=0}^{2^n-1} b_j e^{(n)}_j \in \Hyp_n$. Then $\langle a, b \rangle = \sum\limits_{j=0}^{2^{n-1}-1} a_j b_j - \sum\limits_{j=2^{n-1}}^{2^n-1} a_j b_j$.
\end{itemize}
\end{proposition}

\begin{example}
\leavevmode
\begin{itemize}
	\item The complex numbers $\mathbb{C}$, the quaternions $\mathbb{H}$, the octonions $\mathbb{O}$, and the sedenions $\mathbb{S}$ are the algebras of the main sequence for $n=1,\:2,\:3,$ and $4$, correspondingly, cf.~\cite{baez}.
	\item The split-complex numbers $\hat{\mathbb{C}}$, the split-quaternions $\hat{\mathbb{H}}$, the split-octonions $\hat{\mathbb{O}}$, and the split-sedenions $\hat{\mathbb{S}}$ are the split-algebras for $n=1,\:2,\:3,$ and $4$, correspondingly, see~\cite{bentz, our_split-sedenions}.
\end{itemize}
\end{example}

The following lemma describes the anticentralizer of an arbitrary element of $\A_n$.

\begin{lemma}[{\cite[Lemma 5.8]{our_anticomm}}] \label{lemma:A_n-anticomm}
Let $a \in \A_n$, $a \neq 0$.
\begin{enumerate}
\item If $\Re(a) \neq 0$, $n(a) \neq 0$, then $\Anc_{\A_n}(a) = \{ 0 \}$.
\item If $\Re(a) \neq 0$, $n(a)=0$, then $\Anc_{\A_n}(a) = \mathbb{R}\bar{a}$.
\item If $\Re(a) = 0$, then $\Anc_{\A_n}(a) = \left\{ b \in \A_n \; | \; \Re(b) = 0 \mbox{ and } \langle a,b \rangle = 0 \right\}$. \label{item:A_n-anticomm-pure}
\end{enumerate}
\end{lemma}

We now proceed to some concepts related to associativity. For $a,b,c \in \A$ we denote their associator by $[a,b,c] = (ab)c - a(bc)$, and their anti-associator by $\{ a,b,c \} = (ab)c + a(bc)$. An algebra $\A$ is called {\em flexible} if for all $a,b \in \A$ the equality $[a,b,a] = 0$ holds. Clearly, in a flexible algebra $\A$, we have $[a,b,c]=-[c,b,a]$ for all $a,b,c \in \A$. An algebra $\A$ is called {\em alternative} if for all $a,b \in \A$ the equalities $[a,a,b] = [b,a,a] = 0$ hold.

It is well known that $\A_n$ is alternative if and only if $n \leq 3$. However, $\A_n$ is always flexible, see, e.g.,~\cite[p. 436, Theorem~1]{schafer}.

\begin{definition}[{\cite[p. 12, p. 15]{moreno_alternative}}]
Let $a, b \in \A_n$.
\begin{itemize}
    \item We say that $a$ {\em alternates} with $b$ if $[a,a,b] = 0$.
    \item If $a$ alternates with every $b \in \A_n$, then $a$ is {\em alternative}.
    \item We say that $a$ {\em alternates strongly} with $b$ if $[a,a,b] = 0$ and $[b,b,a] = 0$.
    \item If $a$ alternates strongly with every $b \in \A_n$, then $a$ is {\em strongly alternative}.
\end{itemize}
\end{definition}

\begin{lemma}[{\cite[Lemma 4.8]{our_split-algebras}}] \label{lemma:alternative-elements-are-normed}
Let $a, b \in \A_n$, $a$ alternates with $b$. Then $n(ab) = n(ba) = n(a)n(b)$.
\end{lemma}

\section{Hexagons arising from alternativity}

\subsection{Hexagons in zero divisor graphs}

\begin{notation}
In the statements~\ref{lemma:strongly-alternative-system}~-~\ref{proposition:hexagon-elements} we assume that $a,b \in \A_n$ alternate strongly with $c,d \in \A_n$, and $(a,b)(c,d) = 0$ in $\A_{n+1}$.
\end{notation}

\begin{lemma} \label{lemma:strongly-alternative-system}
The elements $ac,da$ alternate strongly with $a,b,c,d$.
\end{lemma}

\begin{proof}
We have $(a,b)(c,d) = (ac + \gamma_n \bar{d}b, da + b\bar{c}) = 0$, so $ac = -\gamma_n \bar{d}b$ and $da = -b\bar{c}$. Then
\begin{align*}
[a,a,ac] &= -[a,\bar{a},ac] = -(a\bar{a})(ac) + a(\bar{a}(ac)) =\\
&= -(a\bar{a})(ac) + a((\bar{a}a)c) = -n(a)ac + n(a)ac = 0,\\
[b,b,ac] &= [b,b,-\gamma_n \bar{d}b] = \gamma_n[\bar{d}b,b,b] = -\gamma_n[\bar{d}b,\bar{b},b] = 0.
\end{align*}
Similarly, all the elements $a,b,c,d$ alternate with $ac, ad$. Conversely,
\begin{align*}
    [ac,ac,a] &= -[ac, \overline{ac}, a] = -((ac)(\overline{ac}))a + (ac)((\bar{c}\bar{a})a) =\\
    &= -n(ac)a + (ac)(\bar{c}(\bar{a}a)) = -n(ac)a + n(a)(ac)\bar{c} =\\
    &= -n(ac)a + n(a)a(c\bar{c}) = -n(ac)a + n(a)n(c)a = 0,
\end{align*}
since it follows from Lemma~\ref{lemma:alternative-elements-are-normed} that $n(ac) = n(a)n(c)$. Thus it can be shown that $ac, ad$ alternate with $a,b,c,d$.
\end{proof}

\begin{lemma} \label{lemma:A_n-next-pair}
Let $n(c) - \chi \gamma_n n(d) = \chi n(c) - \gamma_n n(d) = 0$ for some $\chi \in \mathbb{R}$. Then $(c,d)(\overline{ac},-\chi da)=0$.
\end{lemma}

\begin{proof}
Indeed, we have
\begin{align*}
    (c,d)(\overline{ac},-\chi da) &= \left( c(\overline{ac}) + \gamma_n (\overline{-\chi da})d, (-\chi da)c + d(ac) \right) =\\
    &= \left(c(\bar{c}\bar{a}) - \chi \gamma_n (\bar{a}\bar{d})d, \chi (b\bar{c})c - \gamma_n d(\bar{d}b) \right) =\\
    &= \left((c\bar{c})\bar{a} - \chi \gamma_n \bar{a}(\bar{d}d), \chi b(\bar{c}c) - \gamma_n (d\bar{d})b \right) =\\
    &= \left((n(c) - \chi \gamma_n n(d)) \bar{a}, (\chi n(c) - \gamma_n n(d)) b \right) = 0. \tag*{\qedhere}
\end{align*}
\end{proof}

\begin{remark} \label{remark:norm-condition}
If $n(c) = n(d) = 0$, then we can take any $\chi \in \mathbb{R}$ in Lemma~\ref{lemma:A_n-next-pair}. Otherwise, we obtain immediately 
\begin{equation*}
\begin{cases}
(n(c))^2 = (n(d))^2 \neq 0;\\
\chi = \gamma_n \dfrac{n(c)}{n(d)} = \gamma_n \dfrac{n(d)}{n(c)} = \pm 1.
\end{cases}
\tag{\textasteriskcentered} \label{equation:norm-condition}
\end{equation*}
\end{remark}

Condition~\eqref{equation:norm-condition} is satisfied automatically if $\A_{n+1}$ is an algebra of the main sequence, see~\cite[pp.~25-27]{moreno}, or if $\A_{n+1}$ is a Cayley--Dickson split-algebra, see~\cite[Lemma~4.1]{our_split-algebras}. In the original statements, $n(c) = n(d)$ if the elements $c$ and $d$ are alternative in $\A_n$. However, the proofs are valid verbatim if $c,d \in \A_n$ alternate (not strongly) with $a,b \in \A_n$. Hence the values of $\chi$ are equal to $-1$ and $1$, respectively.

In an arbitrary real Cayley--Dickson algebra $\A_{n+1}$, condition~\eqref{equation:norm-condition} holds for any element whose components are both standard basis elements up to sign. We will study such elements in Section~\ref{section:basis-pairs}. However, condition~\eqref{equation:norm-condition} is not true in general, see~\cite[Example~4.17]{our_split-algebras}.

\begin{proposition} \label{proposition:self-orthogonality}
If $(a,b) \in \A_{n+1}$ is pure and satisfies condition~\eqref{equation:norm-condition} with $\chi = 1$, then $(a,b)$ is orthogonal to itself.
\end{proposition}

\begin{proof}
By definition, $n((a,b)) = n(a) - \gamma_n n(b) = \gamma_n n(b) - \gamma_n n(b) = 0$, so $(a,b)(a,b) = -(a,b)\overline{(a,b)} = -n((a,b)) = 0$.
\end{proof}

\begin{remark} \label{remark:condition-asterisk}
Let $(c,d)$ satisfy condition~\eqref{equation:norm-condition} and $(n(a))^2 + (n(b))^2 \neq 0$. Assume without loss of generality that $n(b) \neq 0$. By Lemma~\ref{lemma:alternative-elements-are-normed}, we have
$$
n(a)n(d) = n(da) = n(-b\bar{c}) = n(b)n(\bar{c}) = n(b)n(c).
$$
Then $\gamma_n \dfrac{n(a)}{n(b)} = \gamma_n \dfrac{n(c)}{n(d)} = \chi$. Moreover, Lemma~\ref{lemma:alternative-elements-are-normed} implies that
$$
\gamma_n \dfrac{n(\overline{ac})}{n(-\chi da)} = \gamma_n \dfrac{n(ac)}{n(da)} = \gamma_n \dfrac{n(a)n(c)}{n(a)n(d)} = \gamma_n \dfrac{n(c)}{n(d)} = \chi.
$$
Hence $(a,b)$ and $(\overline{ac},-\chi da)$ also satisfy condition~\eqref{equation:norm-condition}.
\end{remark}

\begin{corollary} \label{corollary:A_n-double-hexagon}
Let $(a,b)$ and $(c,d)$ satisfy condition~\eqref{equation:norm-condition}. Then there exists the following $6$-cycle in $\Gamma_Z(\A_{n+1})$:
$$
(a,b) \rightarrow (c,d) \rightarrow (\overline{ac},-\chi da) \rightarrow (a,-b) \rightarrow (c,-d) \rightarrow (\overline{ac}, \chi da) \rightarrow (a,b).
$$
\end{corollary}

\begin{proof}
We obtain this cycle by applying Lemma~\ref{lemma:A_n-next-pair} successively:
\begin{itemize}
    \item $(a,b)(c,d)=0$ implies $(c,d)(\overline{ac},-\chi da)=0$;
    \item $\overline{c(\overline{ac})} = \overline{c(\bar{c}\bar{a})} = \overline{(c\bar{c})\bar{a}} = n(c)a$,\\
    $-\chi (-\chi da)c = (da)c = (-b\bar{c})c = -b(\bar{c}c) = -n(c)b$,\\
    and $n(c) \neq 0$, so $(c,d)(\overline{ac},-\chi da)=0$ implies $(\overline{ac},-\chi da)(a,-b)=0$;
    \item $\overline{(\overline{ac})a} = \overline{(\bar{c}\bar{a})a} = \overline{\bar{c}(\bar{a}a)} = n(a)c$,\\
    $-\chi (-b)(\overline{ac}) = \chi b (\overline{-\gamma_n\bar{d}b}) = -\chi \gamma_n b(\bar{b}d) = -\chi \gamma_n (b\bar{b})d = -\chi \gamma_n n(b)d = -n(a)d$,\\
    and $n(a) \neq 0$, so $(\overline{ac},-\chi da)(a,-b)=0$ implies $(a,-b)(c,-d)=0$;
    \item $(a,-b)(c,-d)=0$ implies $(c,-d)(\overline{ac},\chi da)=0$;
    \item $(c,-d)(\overline{ac},\chi da)=0$ implies $(\overline{ac},\chi da)(a,b)=0$. \qedhere
\end{itemize}
\end{proof}

\begin{proposition} \label{proposition:hexagon-elements}
If it is convenient, we can use that $(\overline{ac},-\chi da) = -\gamma_n(\bar{b}d, -\chi \gamma_n b\bar{c})$ and $(\overline{ac},\chi da) = -\gamma_n(\bar{b}d, \chi \gamma_n b\bar{c})$.
\end{proposition}

\begin{proof}
Since $(a,b)(c,d) = (ac + \gamma_n \bar{d}b, da + b\bar{c}) = 0$, we have $\overline{ac} = -\overline{\gamma_n \bar{d}b} = -\gamma_n \bar{b}d$ and $da = -b\bar{c}$. Hence $(\overline{ac},-\chi da) = (-\gamma_n \bar{b}d, \chi b\bar{c}) = -\gamma_n(\bar{b}d, -\chi \gamma_n b\bar{c})$ and $(\overline{ac},\chi da) = (-\gamma_n \bar{b}d, -\chi b\bar{c}) = -\gamma_n(\bar{b}d, \chi \gamma_n b\bar{c})$.
\end{proof}

\begin{descript}
By using Corollary~\ref{corollary:A_n-double-hexagon} we obtain a subgraph of $\Gamma_Z(\A_{n+1})$ which we call a {\em hexagon}. It is depicted in Fig.~\ref{figure:directed-hexagon}.
\end{descript}

\begin{figure}[H]
\centering
\includegraphics[width=0.54\linewidth]{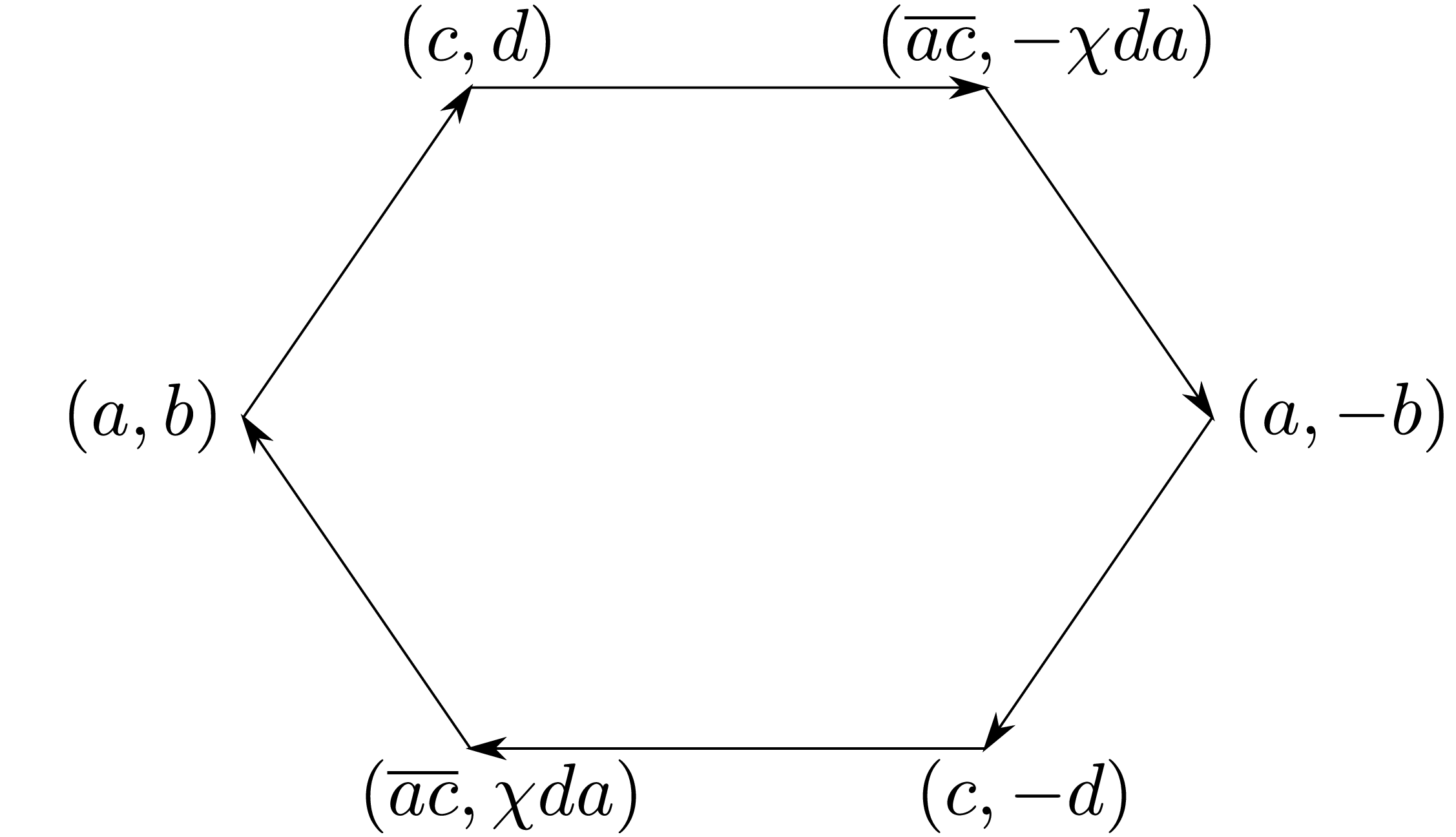}
\caption{\label{figure:directed-hexagon} A hexagon.}
\end{figure}

\subsection{Doubly alternative zero divisors}

By~\cite[Corollary~4.6]{our_split-algebras}, in case of real Cayley--Dickson algebras all zero divisors appear to be two-sided zero divisors, that is, $Z(\A_n) = Z_{LR}(\A_n)$. We now obtain a relation between orthogonality graphs and zero divisor graphs.

\begin{proposition} \label{proposition:orthogonality-condition}
Let $([a], [b])$ be an edge in $\Gamma_Z(\A_n)$. Then $([a], [b])$ is also an edge in $\Gamma_O(\A_n)$ if and only if one of the following conditions holds:
\begin{enumerate}
\item $[b] = [\bar{a}]$ and $n(a) = 0$;
\item $\Re(a) = \Re(b) = 0$.
\end{enumerate}
\end{proposition}

\begin{proof}
Assume that $ab = ba = 0$. Then $ab + ba = 0$, that is, $b \in \Anc_{\A_n}(a)$, so we may apply Lemma~\ref{lemma:A_n-anticomm}. If $\Re(a) = 0$, then $\Re(b) = 0$ and $\langle a, b \rangle = 0$. Otherwise, $n(a) = 0$ and $b \in \mathbb{R}\bar{a}$.

We now prove the converse implication:
\begin{enumerate}
    \item If $n(a) = 0$ and $b \in \Lin(\bar{a})$, then we obtain $ab = ba = 0$ immediately;
    \item If $\Re(a) = \Re(b) = 0$, then $ba = \overline{\bar{a}\bar{b}} = \overline{ab} = \bar{0} = 0$. \qedhere
\end{enumerate}
\end{proof}

It follows from Proposition~\ref{proposition:orthogonality-condition} that any zero divisor $a \in \A_n$ with nontrivial orthogonalizer either is pure or has zero norm. If $a$ is not pure, then its connected component of $\Gamma_O(\A_n)$ consists of two vertices $[a]$ and $[\bar{a}]$. Hence, in the context of orthogonality graphs, we are interested in pure zero divisors only.

\begin{notation}
We denote the set of all pure zero divisors by
$$
Z'(\A_n) = Z(\A_n) \cap \A'_n = \{ x \in Z(\A_n) \; | \; \Re(x) = 0 \}.
$$
Then $\Gamma_O'(\A_n)$ is the subgraph of $\Gamma_O(\A_n)$ on the vertex set $P(Z'(\A_n))$.
\end{notation}

Consider now zero divisors $(a,b)\in \A_{n+1}$ such that both elements $a$ and $b$ are alternative elements in~$\A_n$.

\begin{definition}  The set of {\em doubly alternative elements\/} of $\A_{n+1}$ is
$$
DA(\A_{n+1}) = \{ (a, b) \in \A_{n+1} \; | \; \text{ both } a \text{ and }  b \; \text{are alternative elements in} \; \A_n \}.
$$
\end{definition}

\begin{lemma} \label{lemma:split-algebras-annihilators}
Let $(a,b) \in DA(\A_{n+1})$ satisfy condition~\eqref{equation:norm-condition}. Then
\begin{align*}
l.\Ann_{\A_{n+1}}((a,b)) &= \left\{ \left(c, -\dfrac{(bc)a}{n(a)} \right) \; \bigg| \; b(ca) = \chi (bc)a \right\}, \\
r.\Ann_{\A_{n+1}}((a,b)) &= \left\{ \left(c, -\dfrac{(b\bar{c})\bar{a}}{n(a)} \right) \; \bigg| \; (ac)\bar{b} = \chi a(c\bar{b}) \right\}.
\end{align*}
\end{lemma}

\begin{proof}
We consider $l.\Ann_{\A_{n+1}}((a,b))$ first. Let $(c,d) \in \A_{n+1}$ such that $(c,d)(a,b) = (ca + \gamma_n \bar{b}d, bc + d\bar{a}) = 0$. Then $bc + d\bar{a} = 0$, so $n(a)d = d(\bar{a}a) = (d\bar{a})a = -(bc)a$. Moreover, $ca + \gamma_n\bar{b}d = 0$ and, by condition~\eqref{equation:norm-condition}, $\chi n(a) = \gamma_n n(b)$, hence $b(ca) = -\gamma_n b(\bar{b}d) = -\gamma_n (b\bar{b})d = -\gamma_n n(b)d = -\chi n(a)d = \chi (bc)a$. Reasoning in the opposite way, we may conclude that for any $c \in \A_n$ such that $b(ca) = \chi (bc)a$ we have $\left(c, -\frac{(bc)a}{n(a)} \right) \in l.\Ann_{\A_{n+1}}((a,b))$. So, the converse is also true.

We now proceed to $r.\Ann_{\A_{n+1}}((a,b))$. Let $(c,d) \in \A_{n+1}$ such that $(a,b)(c,d) = (ac + \gamma_n \bar{d}b, da + b\bar{c}) = 0$. Then $da + b\bar{c} = 0$, so $n(a)d = d(a\bar{a}) = (da)\bar{a} = -(b\bar{c})\bar{a}$. Since $ac + \gamma_n \bar{d}b = 0$ and $\chi n(a) = \gamma_n n(b)$, we have $(ac)\bar{b} = -\gamma_n (\bar{d}b)\bar{b} = -\gamma_n \bar{d}(b\bar{b}) = -\gamma_n n(b)\bar{d} = -\chi n(a) \bar{d} = \chi(\overline{-n(a)d}) = \chi \overline{(b\bar{c})\bar{a}} = \chi a(c\bar{b})$. Clearly, the converse is also true in this case, that is, for any $c \in \A_n$ such that $(ac)\bar{b} = \chi a(c\bar{b})$ we have $\left(c, -\frac{(b\bar{c})\bar{a}}{n(a)} \right) \in r.\Ann_{\A_{n+1}}((a,b))$.
\end{proof}

\begin{lemma} \label{lemma:split-algebras-orthogonalizers}
Let $(a,b) \in DA(\A_{n+1})$ be pure and satisfy condition~\eqref{equation:norm-condition}. Then
\begin{align*}
O_{\A_{n+1}}((a,b)) &= \left\{ \left(c, -\dfrac{(bc)a}{n(a)} \right) \; \bigg| \; \Re(c) = 0, \: b(ca) = \chi (bc)a \right\} =\\
&= \left\{ \left(\dfrac{(a\bar{d})b}{n(b)}, d \right) \; \bigg| \; \langle d, ba \rangle = 0, \: a(\bar{d}b) = \chi (a\bar{d})b \right\}.
\end{align*}
\end{lemma}

\begin{proof}
For the first part of the formula, it is sufficient to combine Lemma~\ref{lemma:split-algebras-annihilators} and Proposition~\ref{proposition:orthogonality-condition}.

The second equality is proved similarly. The condition $(a,b)(c,d) = (ac + \gamma_n \bar{d}b, da + b\bar{c}) = 0$ holds if and only if $c = -\frac{(\bar{a}\bar{d})b}{n(b)}$ and $(\bar{b}d)a = \chi \bar{b}(da)$. Since $\Re(a) = 0$, this is equivalent to $c = \frac{(a\bar{d})b}{n(b)}$ and $a(\bar{d}b) = \chi (a\bar{d})b$. Moreover, by Proposition~\ref{proposition:lambda-form} and Lemma~\ref{lemma:inner-product-movement}, we have $\Re(n(b)c) = \langle e_0, n(b)c \rangle = \langle e_0, (a\bar{d})b \rangle = \langle \bar{b}, a\bar{d} \rangle = \langle \bar{a}\bar{b}, \bar{d} \rangle = \langle \overline{ba}, \bar{d} \rangle = \langle d, ba \rangle$, so $\Re(c) = 0$ if and only if $\langle d, ba \rangle = 0$.
\end{proof}

\begin{remark} \label{remark:components-determined}
Lemma~\ref{lemma:split-algebras-orthogonalizers} can also be interpreted as follows. Let $a,b \in \A_n$ alternate strongly with $c,d \in \A_n$, $(a,b)$ and $(c,d)$ be pure, orthogonal in $\A_{n+1}$ and satisfy condition~\eqref{equation:norm-condition}. Then $b(ca) = \chi (bc)a$ and $a(\bar{d}b) = \chi (a\bar{d})b$. Moreover, the value of $d$ is uniquely determined by $a,b,c$ as $d = -\frac{(bc)a}{n(a)}$. The value of $c$ is also uniquely determined by $a,b,d$ as $c = \frac{(a\bar{d})b}{n(b)}$. Note that condition~\eqref{equation:norm-condition} implies that $(n(c))^2 = (n(d))^2 \neq 0$, and thus $c \neq 0$ and $d \neq 0$.
\end{remark}

\begin{corollary} \label{corollary:no-chords-in-hexagons}
Let $a,b,c,d,ac,ad \in \A_n$ alternate strongly pairwise, $(a,b)(c,d) = 0$ in $\A_{n+1}$, $(a,b)$ and $(c,d)$ satisfy condition~\eqref{equation:norm-condition}. Let also $\Re(a) = \Re(c) = \Re(ac) = 0$. Then the hexagon in Fig.~\ref{figure:directed-hexagon} is an undirected hexagon in $\Gamma_O(\A_{n+1})$, and there are no other edges in $\Gamma_O(\A_{n+1})$ which connect its vertices, that is, there are no chords in it.
\end{corollary}

\begin{proof}
Since $\Re(a) = \Re(c) = \Re(ac) = 0$, it follows from Proposition~\ref{proposition:orthogonality-condition} that the hexagon in Fig.~\ref{figure:directed-hexagon} is a hexagon not only in $\Gamma_Z(\A_{n+1})$ but also in $\Gamma_O(\A_{n+1})$. 

Let now $(x,y)$ and $(z,w)$ be two of its vertices, possibly equal. We first show that $(x,y)$ cannot be orthogonal both to $(z,w)$ and to $(z,-w)$. We assume otherwise. Then, by Remark~\ref{remark:components-determined}, $(x,y)(z,w) = 0$ implies $w = -\frac{(yz)x}{n(x)}$, and $(x,y)(z,-w) = 0$ implies $-w = -\frac{(yz)x}{n(x)}$, so $w = -w$. But $w \neq 0$, a contradiction.

Let $\Re(x) = 0$. It is possible in general that $(x,y)(x,-y) = (- n(x) - \gamma_n n(y), -2yx) = 0$, so $(x,y)$ and $(x,-y)$ are orthogonal. However, if $x$ alternates strongly with $y$ and $(x,y)$ satisfies condition~\eqref{equation:norm-condition}, then $(x,y)$ cannot be orthogonal to $(x,-y)$. Indeed, by Lemma~\ref{lemma:alternative-elements-are-normed}, $n(yx) = n(y)n(x) \neq 0$. Hence $yx \neq 0$, and thus $(x,y)(x,-y) \neq 0$.
\end{proof}

\section{Hexagons for algebras of the main sequence} \label{section:main-sequence}

\subsection{Main properties of zero divisors} \label{subsection:main-sequence}

We have already mentioned that $Z(\A_n) = Z_{LR}(\A_n)$. In case of the algebras of the main sequence a stronger result holds.

\begin{lemma}[{\cite[Corollary~1.6, Corollary~1.9, Corollary~1.12]{moreno}}] \label{lemma:moreno-zero-divisor-conditions} \label{lemma:moreno-doubly-pure} \label{lemma:moreno-tilde}
\leavevmode
\begin{itemize}
    \item Let $x,y \in \Main_{n+1}$. Then $xy = 0$ if and only if $yx = 0$.
    \item If $x \in Z(\Main_{n+1})$, then $x$ is doubly pure.
    \item Let $x=(x_1,x_2)$, $y=(y_1,y_2)$, $\til{y}=(-y_2,y_1) \in \Main_{n+1}$. Then $xy=0$ if and only if $x\til{y} = 0$.
\end{itemize}
\end{lemma}

\begin{corollary} \label{corollary:zero-divisors-symmetry}
$\Gamma_Z(\Main_{n+1})$ can be obtained from $\Gamma_O(\Main_{n+1})$ by replacing every undirected edge with a pair of directed edges.
\end{corollary}

\begin{proof}
Follows immediately from Lemma~\ref{lemma:moreno-zero-divisor-conditions}.
\end{proof}

In the following lemma of~\cite{moreno}, Moreno assumes that $a$ and $b$ are alternative elements in $\Main_n$. However, the proof is valid verbatim if we merely suppose that $a,b \in \Main_n$ alternate with $c,d \in \Main_n$.

\begin{lemma}[{\cite[pp. 25-27]{moreno}}] \label{lemma:A_n-alternative-properties}
Let $a,b \in \Main_n$ alternate with $c,d \in \Main_n$, $(a,b),(c,d) \in Z(\Main_{n+1})$, $(a,b)(c,d) = 0$. Then
\begin{enumerate}
\item $\Re(a)=\Re(b)=\Re(c)=\Re(d)=0$;
\item $n(a)=n(b)$;
\item $c \perp d$;
\item $ac = -db$, $da = bc$;
\item $[a,c,b] = 2n(a)d$, $[a,d,b] = -2n(a)c$;
\item $\{ a,c,b \} = \{ a,d,b \} = 0$.
\end{enumerate}
Furthermore, if $a$ alternates strongly with $b$, then
\begin{enumerate}
\setcounter{enumi}{6}
\item $c, d \in \Lin(e_0,a,b,ab)^{\perp}$.
\end{enumerate}
\end{lemma}

\begin{corollary} \label{corollary:A_n-strongly-alternative-properties}
Let $a,b \in \Main_n$ alternate strongly with $c,d \in \Main_n$, $(a,b),(c,d) \in Z(\Main_{n+1})$, $(a,b)(c,d) = 0$. Then
\begin{enumerate}
\item $\Re(a)=\Re(b)=\Re(c)=\Re(d)=0$;
\item $n(a)=n(b)$, $n(c)=n(d)$;
\item $[a,c,b] = 2n(a)d$, $[a,d,b] = -2n(a)c$, $[c,a,d] = 2n(c)b$, $[c,b,d] = -2n(c)a$;
\item $\{ a,c,b \} = \{ a,d,b \} = \{ b,c,a \} = \{ b,d,a \} = 0$,\\
$\{ c,a,d \} = \{ c,b,d \} = \{ d,a,c \} = \{ d,b,c \} = 0$.
\end{enumerate}
If $a,b,c,d$ alternate strongly pairwise, then
\begin{enumerate}
\setcounter{enumi}{4}
\item $a,b,c,d$ are pairwise orthogonal with respect to the inner product $\langle \cdot, \cdot \rangle$;
\item $ac = bd$, $ad = -bc$.
\end{enumerate}
\end{corollary}

\begin{proof}
By Lemma~\ref{lemma:moreno-zero-divisor-conditions}, $(a,b)(c,d)=0$ implies $(c,d)(a,b)=0$, so we can apply Lemma~\ref{lemma:A_n-alternative-properties} twice. Clearly, $\{ b,c,a \} = -\overline{\{ a,c,b \}} = 0$, and the rest of equalities can be proved similarly.

If $a,b,c,d$ alternate strongly pairwise, then $a,b,c,d$ are pairwise orthogonal with respect to the inner product $\langle \cdot, \cdot \rangle$, so Lemma~\ref{lemma:A_n-anticomm} implies that $a,b,c,d$ anticommute pairwise. Hence $ac = -db = bd$ and $ad = -da = -bc$.
\end{proof}

Under the conditions of Corollary~\ref{corollary:A_n-strongly-alternative-properties}, we may always assume without loss of generality that $n(a) = n(b) = n(c) = n(d) = 1$.

\subsection{Double hexagons in orthogonality graphs}

The following lemma describes a product of two elements related to a pair of zero divisors, and it holds for an arbitrary real Cayley--Dickson algebra.

\begin{lemma} \label{lemma:next-mult-pair}
Let $(a,b)(c,d) = 0$ in $\A_{n+1}$. Then
\begin{align*}
    (a,b)(c,-d) &= 2(ac,-da),\\
    (a,-b)(c,d) &= 2(ac,da).
\end{align*}
\end{lemma}

\begin{proof}
Since $(a,b)(c,d) = 0$, we have $da + b\bar{c} = 0$. Hence
\begin{align*}
    (a,b)(c,-d) &= (a,b)(c,-d) + (a,b)(c,d) = (a,b)(2c,0) = 2(ac,b\bar{c}) = 2(ac,-da),\\
    (a,-b)(c,d) &= (a,-b)(c,d) + (a,b)(c,d) = (2a,0)(c,d) = 2(ac,da). \tag*{\qedhere}
\end{align*}
\end{proof}

In the statements~\ref{lemma:orthonormal-system}~-~\ref{remark:linear-independence} and in Fig.~\ref{figure:double-hexagon}
we assume that $a,b \in \Main_n$ alternate strongly with $c,d \in \Main_n$, $(a,b)(c,d) = 0$ in $\Main_{n+1}$, $n(a) = n(b) = n(c) = n(d) = 1$.

\begin{lemma} \label{lemma:orthonormal-system}
The elements $e_0,a,b,c,d,ac,ad$ form an orthonormal system with respect to the inner product $\langle \cdot, \cdot \rangle$.
\end{lemma}

\begin{proof}
By Corollary~\ref{corollary:A_n-strongly-alternative-properties}, $a,b,c,d$ are pure and form an orthonormal system, $ac = bd$, $ad = -bc$. By Lemma~\ref{lemma:alternative-elements-are-normed}, $n(ac) = n(a)n(c) = 1$ and $n(ad) = n(a)n(d) = 1$. Moreover, by Lemma~\ref{lemma:inner-product-movement}, we have
\begin{align*}
\langle ac, ad \rangle &= \langle \bar{a}(ac), d \rangle = \langle (\bar{a}a)c, d \rangle = \langle c, d \rangle = 0,\\
\langle e_0, ac \rangle &= \langle \bar{a}, c \rangle = - \langle a, c \rangle = 0,\\
\langle a, ac \rangle &= \langle \bar{a}a, c \rangle = \langle e_0, c \rangle = 0,\\
\langle b, ac \rangle &= \langle b, bd \rangle = \langle \bar{b}b, d \rangle = \langle e_0, d \rangle = 0.
\end{align*}
Reasoning in the same way, we obtain that $e_0,a,b,c,d,ac,ad$ form an orthonormal system.
\end{proof}

\begin{corollary} \label{corollary:double-hexagon}
There exists the following $6$-cycle in $\Gamma_O(\Main_{n+1})$:
$$
(a,b) \leftrightarrow (c,d) \leftrightarrow (ac,ad) \leftrightarrow (a,-b) \leftrightarrow (c,-d) \leftrightarrow (ac,-ad) \leftrightarrow (a,b).
$$
\end{corollary}

\begin{proof}
We use Corollary~\ref{corollary:A_n-double-hexagon} for $\chi = \gamma_n \frac{n(a)}{n(b)} = -1$. By Lemma~\ref{lemma:orthonormal-system}, we have $\overline{ac} = -ac$ and $da = -ad$. Finally, Corollary~\ref{corollary:zero-divisors-symmetry} implies that directed edges of the hexagon in $\Gamma_Z(\Main_{n+1})$ correspond to undirected ones in $\Gamma_O(\Main_{n+1})$.
\end{proof}

\begin{descript} \label{description:double-hexagon}
By using Lemma~\ref{lemma:moreno-tilde} and Corollary~\ref{corollary:double-hexagon} we obtain a subgraph of $\Gamma_O(\Main_{n+1})$ which we call a {\em double hexagon}. It is depicted in Fig.~\ref{figure:double-hexagon}. A double hexagon consists of six bipartite graphs $K_{2,2}$ glued together.
\end{descript}

\begin{figure}[H]
\centering
\includegraphics[width=0.57\linewidth]{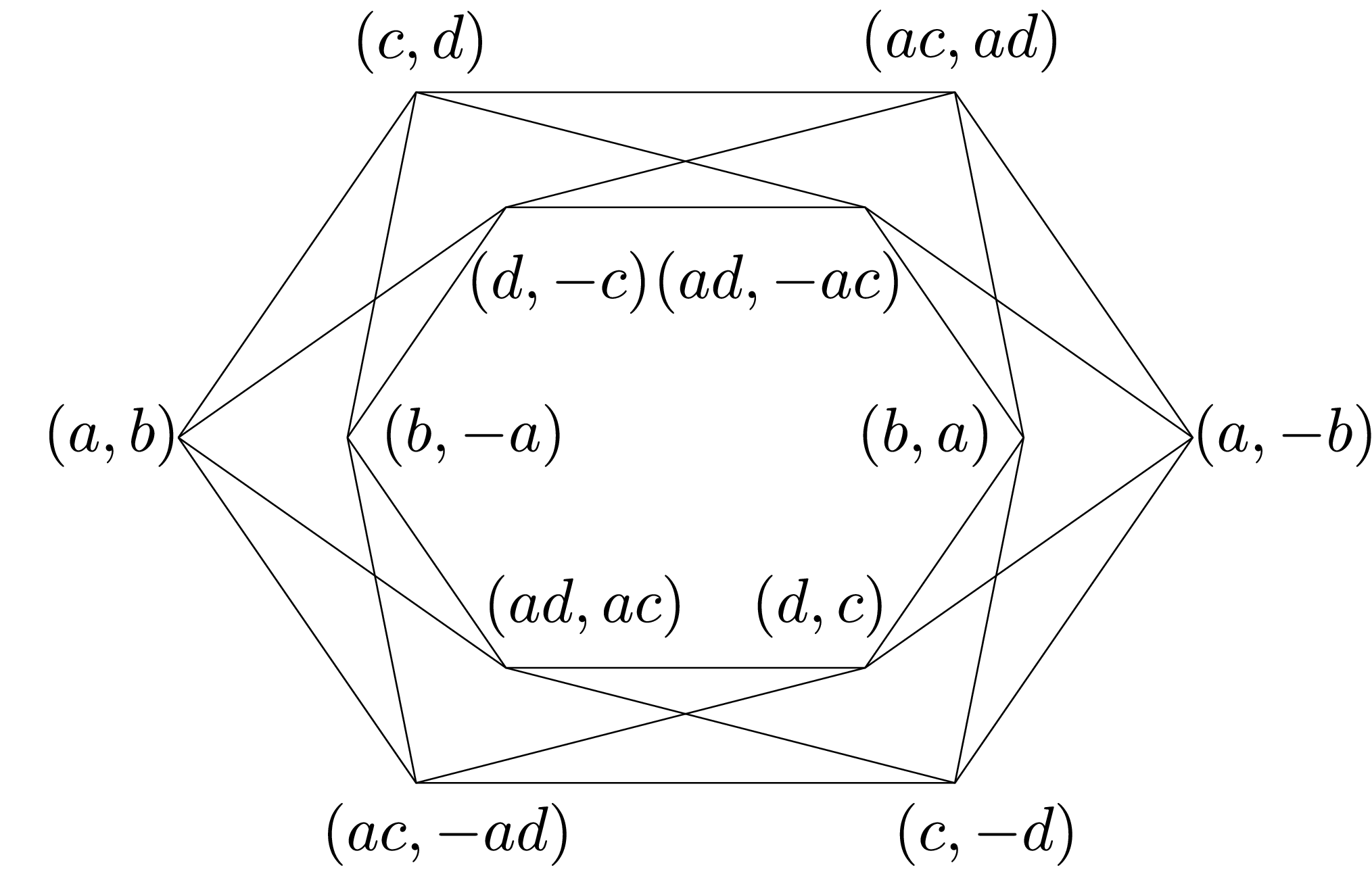}
\caption{\label{figure:double-hexagon} A double hexagon.}
\end{figure}

\begin{remark}
Note that we could start with $(c,d)(a,b) = 0$ instead of $(a,b)(c,d) = 0$ in Corollary~\ref{corollary:double-hexagon}, and we would obtain the same double hexagon but in the counterclockwise direction.
\end{remark}


Our next goal is to find the multiplication table of the vertices of a double hexagon.

\begin{lemma} \label{lemma:products-with-double-hexagon}
The products of $(a,b)$ with the vertices of a double hexagon in Fig.~\ref{figure:double-hexagon} are as follows:
\begin{align*}
    &(a,b)(c,d) = 0, & &(a,b)(c,-d) = 2(ac,ad),\\
    &(a,b)(d,-c) = 0, & &(a,b)(d,c) = 2(ad,-ac),\\
    &(a,b)(ad,ac) = 0, & &(a,b)(ad,-ac) = -2(d,c),\\
    &(a,b)(ac,-ad) = 0, & &(a,b)(ac,ad) = -2(c,-d),\\
    &(a,b)(a,b) = -2(e_0,0), & &(a,b)(a,-b) = 2(0,ab),\\
    &(a,b)(b,-a) = 2(0,e_0), & &(a,b)(b,a) = 2(ab,0).
\end{align*}
\end{lemma}

\begin{proof}
By Lemma~\ref{lemma:orthonormal-system}, $a$ and $b$ are pure, so $a^2 = b^2 = -e_0$, and $a,b,c,d,ac,ad$ anticommute pairwise. It follows from Fig.~\ref{figure:double-hexagon} that $(a,b)(c,d) = (a,b)(d,-c) = (a,b)(ad,ac) = (a,b)(ac,-ad) = 0$. We apply Lemma~\ref{lemma:next-mult-pair} to each of these four equalities and obtain the first four identities in the right column. The four remaining equalities can be verified directly.
\end{proof}

\begin{theorem} \label{theorem:neat-basis}
Let
\begin{align*}
    f_0 &= (e_0,0) = e_0, & f_5 &= (c,d), & \til{f_0} &= (0,e_0), & \til{f_5} &= (-d,c),\\
    f_1 &= (a,b), & f_6 &= (a,-b), & \til{f_1} &= (-b,a), & \til{f_6} &= (b,a),\\
    f_2 &= (c,-d), & f_7 &= (0,ab), & \til{f_2} &= (d,c), & \til{f_7} &= (-ab,0),\\
    f_3 &= (ac,ad), & f_8 &= (0,dc), & \til{f_3} &= (-ad,ac), & \til{f_8} &= (-dc,0),\\
    f_4 &= (ac,-ad), & f_9 &= (0,(ac)(ad)), & \til{f_4} &= (ad,ac), & \til{f_9} &= (-(ac)(ad),0).
\end{align*}
Then $F = \{ f_m, \til{f_m} \: | \: m = 0, \dots, 6 \}$ is an orthogonal system in $\Main_{n+1}$ with respect to the inner product $\langle \cdot, \cdot \rangle$. Its multiplication table is given in Table~\ref{table:neat-mult}. The elements $f_7,f_8,f_9,\til{f_7},\til{f_8},\til{f_9}$ are auxiliary here.
\end{theorem}

\begin{proof}
The element $f_0$ is the unity of $\Main_{n+1}$. By Lemma~\ref{lemma:orthonormal-system}, $e_0,a,b,c,d,ac,ad$ form an orthonormal system in $\Main_n$ with respect to the inner product $\langle \cdot, \cdot \rangle$. Hence $F$ is an orthogonal system in $\Main_{n+1}$. It follows from Lemma~\ref{lemma:A_n-anticomm} that the elements of $F \setminus \{ f_0 \}$ anticommute pairwise.

By~\cite[p.~13]{moreno}, for any $(x,y) \in \mathbb{S}$ we have $(x,y)\til{f_0} = (x,y)(0,e_0) = (-y,x)$. Hence $f_m \til{f_0} = \til{f_m}$ and $\til{f_m} \til{f_0} = -f_m$ for all $m = 0, \dots, 9$.

For the rest of the table, we use Lemma~\ref{lemma:products-with-double-hexagon}. The elements $(a,b)$ and $(c,d)$ can be replaced with any other pair of adjacent vertices of a double hexagon, since they represent an arbitrary pair of zero divisors in $\Main_{n+1}$ with the condition of strong alternativity.
\end{proof}

\begin{remark}
According to Lemma~\ref{lemma:moreno-tilde}, the elements $f_m$ and $-\til{f_m}$ belong to the same corner of a double hexagon, $m = 1, \dots, 6$. The element $f_m$ belongs to the outer hexagon, while $-\til{f_m}$ belongs to the inner hexagon.

The elements $f_1$, $f_2$ and $f_3$ form a set of non-adjacent vertices of the outer hexagon, as well as $f_4$, $f_5$ and $f_6$ do.
\end{remark}

\begin{remark} \label{remark:linear-independence}
It follows from Theorem~\ref{theorem:neat-basis} that all elements in the vertices of the double hexagon in Fig.~\ref{figure:double-hexagon} are linearly independent.
\end{remark}

One can see that Table~\ref{table:neat-mult} has a block structure, the blocks being of two types. Some of them are antidiagonal, and the others resemble the multiplication table of the unit quaternions. However, we cannot scale the elements in such a way that one of the blocks becomes a true quaternion table. In particular, the subalgebra $\Lin(f_0,f_1,f_2,f_3)$ is not associative, and thus not isomorphic to $\mathbb{H}$, see~\cite[p.~496]{chan}.

\begin{landscape}
\begin{table}[H]
\centering
$
\begin{array}{|c||c:ccc:ccc:c:ccc:ccc|}
\hline
\vphantom{\Big|} \times & \hphantom{.} f_0 \hphantom{.} & f_1 & f_2 & f_3 & f_4 & f_5 & f_6 & \til{f_0} & \til{f_1} & \til{f_2} & \til{f_3} & \til{f_4} & \til{f_5} & \til{f_6}\\
\hline\hline
\vphantom{\Big|} \hphantom{.} f_0 \hphantom{.} & f_0 & f_1 & f_2 & f_3 & f_4 & f_5 & f_6 & \til{f_0} & \til{f_1} & \til{f_2} & \til{f_3} & \til{f_4} & \til{f_5} & \til{f_6}\\
\hdashline
\vphantom{\Big|} f_1 & f_1 & -2 f_0 & 2 f_3 & -2 f_2 & 0 & 0 & 2 f_7 & \til{f_1} & -2 \til{f_0} & -2 \til{f_3} & 2 \til{f_2} & 0 & 0 & -2 \til{f_7}\\
\vphantom{\Big|} f_2 & f_2 & -2 f_3 & -2 f_0 & 2 f_1 & 0 & 2 f_8 & 0 & \til{f_2} & 2 \til{f_3} & -2 \til{f_0} & -2 \til{f_1} & 0 & -2 \til{f_8} & 0\\
\vphantom{\Big|} f_3 & f_3 & 2 f_2 & -2 f_1 & -2 f_0 & 2 f_9 & 0 & 0 & \til{f_3} & -2 \til{f_2} & 2 \til{f_1} & -2 \til{f_0} & -2 \til{f_9} & 0 & 0\\
\hdashline
\vphantom{\Big|} f_4 & f_4 & 0 & 0 & -2 f_9 & -2 f_0 & -2 f_6 & 2 f_5 & \til{f_4} & 0 & 0 & 2 \til{f_9} & -2 \til{f_0} & 2 \til{f_6} & -2 \til{f_5}\\
\vphantom{\Big|} f_5 & f_5 & 0 & -2 f_8 & 0 & 2 f_6 & -2 f_0 & -2 f_4 & \til{f_5} & 0 & 2 \til{f_8} & 0 & -2 \til{f_6} & -2 \til{f_0} & 2 \til{f_4}\\
\vphantom{\Big|} f_6 & f_6 & -2 f_7 & 0 & 0 & -2 f_5 & 2 f_4 & -2 f_0 & \til{f_6} & 2 \til{f_7} & 0 & 0 & 2 \til{f_5} & -2 \til{f_4} & -2 \til{f_0}\\
\hdashline
\vphantom{\Big|} \til{f_0} & \til{f_0} & -\til{f_1} & -\til{f_2} & -\til{f_3} & -\til{f_4} & -\til{f_5} & -\til{f_6} & -f_0 & f_1 & f_2 & f_3 & f_4 & f_5 & f_6\\
\hdashline
\vphantom{\Big|} \til{f_1} & \til{f_1} & 2 \til{f_0} & -2 \til{f_3} & 2 \til{f_2} & 0 & 0 & -2 \til{f_7} & -f_1 & -2 f_0 & -2 f_3 & 2 f_2 & 0 & 0 & -2 f_7\\
\vphantom{\Big|} \til{f_2} & \til{f_2} & 2 \til{f_3} & 2 \til{f_0} & -2 \til{f_1} & 0 & -2 \til{f_8} & 0 & -f_2 & 2 f_3 & -2 f_0 & -2 f_1 & 0 & -2 f_8 & 0\\
\vphantom{\Big|} \til{f_3} & \til{f_3} & -2 \til{f_2} & 2 \til{f_1} & 2 \til{f_0} & -2 \til{f_9} & 0 & 0 & -f_3 & -2 f_2 & 2 f_1 & -2 f_0 & -2 f_9 & 0 & 0\\
\hdashline
\vphantom{\Big|} \til{f_4} & \til{f_4} & 0 & 0 & 2 \til{f_9} & 2 \til{f_0} & 2 \til{f_6} & -2 \til{f_5} & -f_4 & 0 & 0 & 2 f_9 & -2 f_0 & 2 f_6 & -2 f_5\\
\vphantom{\Big|} \til{f_5} & \til{f_5} & 0 & 2 \til{f_8} & 0 & -2 \til{f_6} & 2 \til{f_0} & 2 \til{f_4} & -f_5 & 0 & 2 f_8 & 0 & -2 f_6 & -2 f_0 & 2 f_4\\
\vphantom{\Big|} \til{f_6} & \til{f_6} & 2 \til{f_7} & 0 & 0 & 2 \til{f_5} & -2 \til{f_4} & 2 \til{f_0} & -f_6 & 2 f_7 & 0 & 0 & 2 f_5 & -2 f_4 & -2 f_0\\
\hline
\end{array}
$
\caption{\label{table:neat-mult} Multiplication table of the vertices of a double hexagon.}
\end{table}
\end{landscape}

\section{Alternative subalgebras} \label{section:alternative-subalgebras}

Similarly to~\cite{moreno, moreno_alternative, moreno_constructing}, we denote $\til{e}_0 = (0,e_0) \in \A_n$ and $\til{a} = a \til{e}_0$ for all $a \in \A_n$. We have already used a particular case of this notation in Lemma~\ref{lemma:moreno-tilde}. We also denote $\Lin(x_1, \dots, x_k) = \mathbb{R} x_1 + \dots + \mathbb{R} x_k$.

\begin{lemma} \label{lemma:tilde-properties}
Let $a,b \in \A_n$, and $b$ be doubly pure. Then
\begin{enumerate}
    \item $\til{\til{a}} = \gamma_{n-1}a$;
    \item $\til{a}b = -\til{ab}$;
\end{enumerate}
If $a$ is also doubly pure, then
\begin{enumerate}
\setcounter{enumi}{2}
    \item $\til{a} \perp a$;
    \item $\til{a}b + \til{b}a = 0$ if and only if $a \perp b$;
    \item $\gamma_{n-1}ab + \til{b}\til{a} = 0$ if and only if $\til{a} \perp b$.
\end{enumerate}
\end{lemma}

\begin{proof}
Let $a=(a_1,a_2)$, $b=(b_1,b_2)$. By definition, $\til{a} = (a_1,a_2)(0,e_0) = (\gamma_{n-1}a_2,a_1)$.
\begin{enumerate}
\item It holds $\til{\til{a}} = \til{(\gamma_{n-1}a_2,a_1)} = (\gamma_{n-1}a_1, \gamma_{n-1}a_2) = \gamma_{n-1}a$.
\item Since $b$ is doubly pure, we have 
    \begin{multline*}
    \til{a}b = (\gamma_{n-1}a_2,a_1)(b_1,b_2) = (\gamma_{n-1}a_2b_1 + \gamma_{n-1}\bar{b}_2a_1, \gamma_{n-1}b_2a_2 + a_1\bar{b}_1) =\\
    = -(\gamma_{n-1}(b_2a_1 + a_2\bar{b}_1), a_1b_1 + \gamma_{n-1}\bar{b}_2a_2) = -\til{ab}.
    \end{multline*}
\item Since $a$ is doubly pure, $\til{a} = a\til{e}_0 = -\til{e}_0a$, so $\til{a}a + a\til{a} = (a\til{e}_0)a - a(\til{e}_0 a) = [a,\til{e}_0,a] = 0$ by flexibility.
\item By Lemma~\ref{lemma:A_n-anticomm}, $a \perp b$ if and only if $ab = -ba$, which is equivalent to $-\til{a}b = \til{ab} = -\til{ba} = \til{b}a$.
\item By Lemma~\ref{lemma:A_n-anticomm}, $\til{a} \perp b$ if and only if $\til{a}b = -b\til{a}$ or, equivalently, $-\gamma_{n-1}ab = - \til{\til{ab}} = \til{\til{a}b} = -\til{b\til{a}} = \til{b}\til{a}$. \qedhere
\end{enumerate}
\end{proof}

\begin{corollary}
It follows from the explicit form of $\til{a}$ that for $e_j \in \A_n$ we have
$$
\til{e_j} = 
\begin{cases}
e_{j+2^{n-1}}, & 0 \leq j \leq 2^{n-1}-1,\\
\gamma_{n-1}e_{j-2^{n-1}}, & 2^{n-1} \leq j \leq 2^n-1.
\end{cases}
$$
\end{corollary}

Note that in Lemmas~\ref{lemma:quaternionic-subalgebra},~\ref{lemma:nonspecial-quaternionic-subalgebra}, and~\ref{lemma:octonionic-subalgebra} we allow $n(a)$ and $n(b)$ to be equal to zero, in contrast to the usual definition of Cayley--Dickson algebras.

\begin{lemma} \label{lemma:quaternionic-subalgebra}
Let $a \in \A_n$ be doubly pure. Consider $\mathbb{H}_{a} = \Lin(e_0,a,\til{e}_0,\til{a})$. Then there exists a surjective homomorphism $\phi_a: \A_2 \{ -n(a), \gamma_{n-1}\} \to \mathbb{H}_{a}$, so $\mathbb{H}_{a}$ is associative.
\end{lemma}

\begin{proof}
We denote $\mu_1 = n(a)$, $\mu_2 = n(\til{e}_0) = -\gamma_{n-1}$. Since $a$ and $\til{e}_0$ are pure, we have $a^2 = -n(a) = -\mu_1$ and $\til{e}_0^2 = -n(\til{e}_0) = -\mu_2$. The condition $a \in \Lin(e_0,\til{e}_0)^{\perp}$ implies $\til{a} \in \Lin(e_0,\til{e}_0)^{\perp}$. By Lemma~\ref{lemma:tilde-properties}(3), $\til{a} \perp a$, so $a, \til{e}_0, \til{a}$ anticommute pairwise. It remains to show that $\til{a}a = -\til{aa} = \til{\mu_1 e_0} = \mu_1 \til{e}_0$ by Lemma~\ref{lemma:tilde-properties}(2), $\til{a}\til{e}_0 = \til{\til{a}} = \gamma_{n-1} a = -\mu_2 a$ by Lemma~\ref{lemma:tilde-properties}(1), and $\til{a}\til{a} = -\til{a\til{a}} = \til{\mu_1 \til{e}_0} = \mu_1 \gamma_{n-1} = - \mu_1\mu_2$ by Lemma~\ref{lemma:tilde-properties}(1-2). Hence the multiplication table in $\mathbb{H}_{a}$ is given by Table~\ref{table:quaternionic-subalgebra}.

Now we may define $\phi_a: \A_2 \{ -\mu_1, -\mu_2\} \to \mathbb{H}_a$ by $\phi_a(e_0) = e_0$, $\phi_a(e_1) = a$, $\phi_a(e_2) = \til{e}_0$, $\phi_a(e_3) = \til{a}$. The multiplication table in Fig.~\ref{table:quaternionic-subalgebra} coincides with the multiplication table of $\A_2\{ -\mu_1, -\mu_2\}$, and $e_0,e_1,e_2,e_3$ form a basis in $\A_2\{ -\mu_1, -\mu_2\}$. Hence every nontrivial relation in $\A_2 \{ -\mu_1, -\mu_2\}$ is preserved under $\phi_a$, so $\phi_a$ is indeed a homomorphism. Clearly, $\phi_a$ is surjective, since $\mathbb{H}_a = \Lin(e_0,a,\til{e}_0,\til{a})$.
\end{proof}

\begin{remark}
Note that $\phi_a$ in Lemma~\ref{lemma:quaternionic-subalgebra} might have nontrivial kernel even for $a \neq 0$, since it is possible that $a = \til{a}$.
\end{remark}

\begin{corollary}
The element $\til{e}_0$ is strongly alternative in $\A_n$.
\end{corollary}

\begin{proof}
Let $a \in \A_n$, $a'$ be the orthogonal projection of $a$ onto $\Lin(e_0, \til{e}_0)^{\perp}$. By Lemma~\ref{lemma:quaternionic-subalgebra}, $a'$ and $\til{e}_0$ generate an associative subalgebra $\mathbb{H}_{a'} \subset \A_n$. Clearly, $a \in \mathbb{H}_{a'}$, so $[a,a,\til{e}_0] = [\til{e}_0,\til{e}_0,a] = 0$.
\end{proof}

\begin{table}[H]
\centering
\begin{minipage}{0.5\textwidth}
\centering
$
\begin{array}{|c|cccc|}
\hline
\times          & e_0             & a                     & \til{e}_0 & \til{a}          \\\hline
e_0             & e_0             & a                     & \til{e}_0 & \til{a}          \\
a               & a               & -\mu_1                & \til{a}   & -\mu_1 \til{e}_0 \\
\hphantom{.} \til{e}_0 \hphantom{.} & \hphantom{.} \til{e}_0 \hphantom{.} & -\til{a}        & -\mu_2          & \mu_2 a                \\
\til{a}   & \til{a}   & \hphantom{.} \mu_1 \til{e}_0 \hphantom{.} & \hphantom{.} -\mu_2 a \hphantom{.} & \hphantom{.} -\mu_1 \mu_2 \hphantom{.} \\\hline
\end{array}
$
\caption{\label{table:quaternionic-subalgebra} Multiplication table in $\mathbb{H}_a$.}
\end{minipage}%
\begin{minipage}{0.5\textwidth}
\centering
$
\begin{array}{|c|cccc|}
\hline
\times & e_0 & a        & b        & ab           \\\hline
e_0    & e_0 & a        & b        & ab           \\
a      & a   & -\mu_1   & ab       & -\mu_1 b     \\
b      & b   & \hphantom{.} -ab \hphantom{.} & -\mu_2   & \mu_2 a      \\
\hphantom{.} ab \hphantom{.} & \hphantom{.} ab \hphantom{.} & \mu_1 b  & \hphantom{.} -\mu_2 a \hphantom{.} & \hphantom{.} -\mu_1 \mu_2 \hphantom{.} \\\hline
\end{array}
$
\caption{\label{table:nonspecial-quaternionic-subalgebra} Multiplication table in $\mathbb{H}_{a,b}$.}
\end{minipage}
\end{table}

\begin{lemma} \label{lemma:nonspecial-quaternionic-subalgebra}
Let $a,b \in \A_n$ be pure, $b \perp \Lin(e_0,a)$. Let also $a$ alternate strongly with~$b$. We denote $\mathbb{H}_{a,b} = \Lin(e_0,a,b,ab)$. Then there exists a surjective homomorphism $\psi_{a,b}: \A_2 \{ -n(a), -n(b) \} \to \mathbb{H}_{a,b}$, and thus $\mathbb{H}_{a,b}$ is associative.
\end{lemma}

\begin{proof}
We denote $\mu_1 = n(a)$, $\mu_2 = n(b)$. Since $a$ and $b$ are pure, we have $a^2 = -n(a) = -\mu_1$ and $b^2 = -n(b) = -\mu_2$. It follows from Lemma~\ref{lemma:inner-product-movement} that $e_0,a,b,ab$ form an orthogonal system. Then, by Lemma~\ref{lemma:A_n-anticomm}, $a,b,ab$ anticommute pairwise, so $ab$ is also pure.

The element $a$ alternates strongly with $b$, so $a(ab) = a^2b = -\mu_1 b$, $b(ab) = -b(ba) = -b^2a = \mu_2 a$, $(ab)a = -(ba)a = -ba^2 = \mu_1 b$, $(ab)b = ab^2 = -\mu_2 a$. Finally, it follows from Lemma~\ref{lemma:alternative-elements-are-normed} that $(ab)^2 = -n(ab) = -n(a)n(b) = -\mu_1 \mu_2$. Hence the multiplication table in $\mathbb{H}_{a,b}$ is given by Table~\ref{table:nonspecial-quaternionic-subalgebra}.

Then we may define $\psi_{a,b}$ by $\psi_{a,b}(e_0) = e_0$, $\psi_{a,b}(e_1) = a$, $\psi_{a,b}(e_2) = b$, $\psi_{a,b}(e_3) = ab$. The rest of the proof is similar to that of Lemma~\ref{lemma:quaternionic-subalgebra}.
\end{proof}

\begin{remark}
The mapping $\psi_{a,b}$ in Lemma~\ref{lemma:nonspecial-quaternionic-subalgebra} might have nontrivial kernel even for $a \neq 0$ and $b \neq 0$, since it is possible that $ab = 0$ or $a,b,ab$ are linearly dependent.
\end{remark}

The next lemma generalizes~\cite[Theorem 5.1]{moreno_alternative}.

\begin{lemma} \label{lemma:octonionic-subalgebra}
Let $a,b \in \A_n$ be doubly pure, $b \perp \Lin(e_0,a,\til{e}_0,\til{a})$. Let also $a$ alternate strongly with $b$. We denote $\mathbb{O}_{a,b} = \Lin(e_0,a,b,ab,\til{e}_0,\til{a},\til{b},\til{ab})$. Then there exists a surjective homomorphism $\phi_{a,b}: \A_3 \{ -n(a), -n(b), \gamma_{n-1}\} \to \mathbb{O}_{a,b}$, so $\mathbb{O}_{a,b}$ is alternative.
\end{lemma}

\begin{proof}
We denote $\mu_1 = n(a)$, $\mu_2 = n(b)$, $\mu_3 = n(\til{e}_0) = -\gamma_{n-1}$. Since $a, b$ and $\til{e}_0$ are pure, we have $a^2 = -n(a) = -\mu_1$, $b^2 = -n(b) = -\mu_2$, and $\til{e}_0^2 = -n(\til{e}_0) = -\mu_3$. We may use Lemma~\ref{lemma:inner-product-movement} to show that $a \in \Lin(e_0,\til{e}_0)^{\perp}$ and $b \perp \Lin(e_0,a,\til{e}_0,\til{a})$ imply that $e_0,a,b,ab,\til{e}_0,\til{a},\til{b},\til{ab}$ form an orthogonal system. Then, by Lemma~\ref{lemma:A_n-anticomm}, $a,b,ab,\til{e}_0,\til{a},\til{b},\til{ab}$ anticommute pairwise. Note that $ab$ is also doubly pure.

By Lemma~\ref{lemma:nonspecial-quaternionic-subalgebra}, there exists a surjective homomorphism $\psi_{a,b}: \A_2 \{ -\mu_1, -\mu_2 \} \to \mathbb{H}_{a,b}$. We now extend it to $\phi_{a,b}: \A_3 \{ -\mu_1, -\mu_2, -\mu_3 \} \to \mathbb{O}_{a,b}$. We may apply Lemma~\ref{lemma:quaternionic-subalgebra} to $a,b$ and $ab$ independently. We then use Lemma~\ref{lemma:tilde-properties}(2) to obtain that $\til{a}b = -\til{ab}$, $\til{a}(ab) = -\til{a(ab)} = \mu_1 \til{b}$, $\til{b}a = -\til{ba} = \til{ab}$, $\til{b}(ab) = -\til{b(ab)} = -\mu_2 \til{a}$, $\til{ab} \cdot a = -\til{(ab)a} = -\mu_1\til{b}$, $\til{ab} \cdot b = -\til{(ab)b} = \mu_2\til{a}$. We use Lemma~\ref{lemma:tilde-properties}(1-2) to get that $\til{b}\til{a} = -\til{b\til{a}} = -\til{\til{ab}} = -\gamma_{n-1}ab = \mu_3 ab$, $\til{ab} \cdot \til{a} = -\til{(ab)\til{a}} = \mu_1 \til{\til{b}} = \mu_1 \gamma_{n-1} b = -\mu_1 \mu_3 b$, and $\til{ab} \cdot \til{b} = -\til{(ab)\til{b}} = -\mu_2 \til{\til{a}} = -\mu_2 \gamma_{n-1} a = \mu_2 \mu_3 a$. Therefore, the multiplication table in $\mathbb{O}_{a,b}$ is given by Table~\ref{table:octonionic-subalgebra}.

Hence we may define $\phi_{a,b}$ by $\phi_{a,b}((e_j,0)) = \psi_{a,b}(e_j)$ and $\phi_{a,b}((0,e_j)) = \til{\psi_{a,b}(e_j)}$ for $0 \leq j \leq 3$. The rest of the proof is similar to that of Lemma~\ref{lemma:quaternionic-subalgebra}.
\end{proof}

\begin{table}[ht]
\centering
$
\begin{array}{|c|cccccccc|}
\hline
\times          & e_0             & a                     & b                     & ab                          & \til{e}_0  & \til{a}          & \til{b}          & \til{ab}               \\\hline
e_0             & e_0             & a                     & b                     & ab                          & \til{e}_0  & \til{a}          & \til{b}          & \til{ab}               \\
a               & a               & -\mu_1                & ab                    & -\mu_1 b                    & \til{a}    & -\mu_1 \til{e}_0 & -\til{ab}        & \mu_1 \til{b}          \\
b               & b               & -ab                   & -\mu_2                & \mu_2 a                     & \til{b}    & \til{ab}         & -\mu_2 \til{e}_0 & -\mu_2 \til{a}         \\
ab              & ab              & \mu_1 b               & \hphantom{.} -\mu_2 a \hphantom{.} & -\mu_1 \mu_2                & \til{ab}   & -\mu_1 \til{b}   & \mu_2 \til{a}    & -\mu_1 \mu_2 \til{e}_0 \\
\til{e}_0 & \til{e}_0 & -\til{a}        & -\til{b}        & -\til{ab}             & -\mu_3           & \mu_3 a                & \mu_3 b                & \mu_3 ab                     \\
\til{a}   & \til{a}   & \mu_1 \til{e}_0 & -\til{ab}       & \mu_1 \til{b}         & -\mu_3 a         & -\mu_1 \mu_3           & \hphantom{.} -\mu_3 ab \hphantom{.} & \mu_1 \mu_3 b                \\
\til{b}   & \til{b}   & \til{ab}        & \mu_2 \til{e}_0 & -\mu_2 \til{a}        & -\mu_3 b         & \mu_3 ab               & -\mu_2 \mu_3           & -\mu_2 \mu_3 a               \\
\hphantom{.} \til{ab} \hphantom{.} & \hphantom{.} \til{ab} \hphantom{.} & \hphantom{.} -\mu_1 \til{b} \hphantom{.} & \mu_2 \til{a}   & \hphantom{.} \mu_1 \mu_2 \til{e}_0 \hphantom{.} & \hphantom{.} -\mu_3 ab \hphantom{.} & \hphantom{.} -\mu_1 \mu_3 b \hphantom{.} & \mu_2 \mu_3 a          & \hphantom{.} -\mu_1 \mu_2 \mu_3 \hphantom{.} \\\hline
\end{array}
$
\caption{\label{table:octonionic-subalgebra} Multiplication table in $\mathbb{O}_{a,b}$.}
\end{table}

\begin{corollary} \label{corollary:alternative-subalgebras}
Let $x,y \in \A_{n-1}$ alternate strongly. Then in $\A_n$ we have
$$
x\til{y} = \til{yx}, \quad \til{x}y = \til{x\bar{y}}, \quad \til{x}\til{y} = \gamma_{n-1}\bar{y}x.
$$
\end{corollary}

\begin{proof}
Let $x' = \Im(x)$ and $y'$ be the orthogonal projection of $\Im(y)$ onto $\Lin(x')^{\perp}$. Clearly, $x'$ and $y'$ are doubly pure in $\A_n$ and alternate strongly, and also $y' \perp \Lin(e_0,x',\til{e}_0,\\\til{x'})$. By Lemma~\ref{lemma:octonionic-subalgebra}, $x'$, $y'$ and $\til{e}_0$ generate an alternative subalgebra $\mathbb{O}_{x',y'} \subset \A_n$. Hence, by~\cite[pp.~419-420]{moufang}, $\mathbb{O}_{x',y'}$ satisfies the Moufang identities, that is, for any $a,b,c \in \mathbb{O}_{x',y'}$ we have
\begin{align}
c(a(cb)) &= ((ca)c)b; \label{equation:moufang1}\\
a(c(bc)) &= ((ac)b)c; \label{equation:moufang2}\\
(ca)(bc) &= (c(ab))c; \label{equation:moufang3}\\
(ca)(bc) &= c((ab)c). \label{equation:moufang4}
\end{align}
Note that $x, y \in \mathbb{O}_{x',y'}$. For any $z \in \A_{n-1}$ we have $\til{z} = z\til{e}_0 = \til{e}_0\bar{z}$. Then
\begin{align*}
    x\til{y} &= x(\til{e}_0\bar{y}) = \dfrac{1}{\til{e}_0^2}\til{e}_0(\til{e}_0(x(\til{e}_0\bar{y}))) \numeq{\ref{equation:moufang1}} \dfrac{1}{\til{e}_0^2}\til{e}_0(((\til{e}_0x)\til{e}_0)\bar{y}) =\\
    &\qquad\qquad\, = \dfrac{1}{\til{e}_0^2}\til{e}_0(((\bar{x}\til{e}_0)\til{e}_0)\bar{y}) = \dfrac{\til{e}_0^2}{\til{e}_0^2}\til{e}_0(\bar{x}\bar{y}) = \til{e}_0(\overline{yx}) = \til{yx},\\
    \til{x}y &= (x\til{e}_0)y = \dfrac{1}{\til{e}_0^2}(((x\til{e}_0)y)\til{e}_0)\til{e}_0 \numeq{\ref{equation:moufang2}} \dfrac{1}{\til{e}_0^2}(x(\til{e}_0(y\til{e}_0)))\til{e}_0 =\\
    &\qquad\qquad\, = \dfrac{1}{\til{e}_0^2}(x(\til{e}_0(\til{e}_0\bar{y})))\til{e}_0 = \dfrac{\til{e}_0^2}{\til{e}_0^2}(x\bar{y})\til{e}_0 = \til{x\bar{y}},\\
    \til{x}\til{y} &= (x\til{e}_0)(y\til{e}_0) = (\til{e}_0\bar{x})(y\til{e}_0) \numeq{\ref{equation:moufang3}} (\til{e}_0(\bar{x}y))\til{e}_0 = ((\overline{\bar{x}y})\til{e}_0)\til{e}_0 = (\bar{y}x)\til{e}_0^2 = \gamma_{n-1}\bar{y}x. \tag*{\qedhere}
\end{align*}
\end{proof}

\section{Zero divisors of the form $\left[\left(e_i^{(n)}, \pm e_j^{(n)}\right)\right]$} \label{section:basis-pairs}

\subsection{Zero divisor classification}

\begin{notation}
We will need the following subsets of $Z(\A_{n+1})$:
\begin{align*}
Z_e(\A_{n+1}) &= \left\{ \left(e_i^{(n)}, \pm e_j^{(n)}\right) \in Z(\A_{n+1}) \right\} = (\E_n \times (\pm \E_n)) \cap Z(\A_{n+1});\\
Z_e'(\A_{n+1}) &= \left\{ \left(e_i^{(n)}, \pm e_j^{(n)}\right) \in Z(\A_{n+1}) \; \big| \; i \neq 0 \right\} = (\E'_n \times (\pm \E_n)) \cap Z(\A_{n+1}).
\end{align*}

Then $\Gamma_e(\A_{n+1})$ denotes a subgraph of $\Gamma_O'(\A_{n+1})$ on the vertex set $P(Z_e'(\A_{n+1}))$.
\end{notation}

For simplicity of notation, below we define zero divisors of $\A_{n+1}$ up to multiplication by $\pm 1$.
By Proposition~\ref{proposition:orthogonality-condition}, in order to prove that $a,b \in Z_e'(\A_{n+1})$ are orthogonal, it is sufficient to show that $ab = 0$, that is, we verify that only one product is equal to zero.

\begin{lemma} \label{lemma:basis-pairs-zero-divisors-relation}
Let $(a,b), (c,d) \in Z_e(\A_{n+1})$, $(a,b)(c,d) = 0$. Then $ab = \pm cd$.
\end{lemma}

\begin{proof}
By~\cite[Lemma~4]{schafer}, $e_j^{(n)}$ is alternative in $\A_n$. Moreover, for any elements $e_j^{(n)}$ and $e_k^{(n)}$ we always have $\overline{e_j^{(n)}} = \pm e_j^{(n)}$, $n(e_j^{(n)}) = \pm 1$ and $e_j^{(n)} e_k^{(n)} = \pm e_k^{(n)} e_j^{(n)}$.

The condition $(a,b)(c,d) = (ac + \gamma_n \bar{d}b, da + b\bar{c}) = 0$ implies $0 = ac + \gamma_n \bar{d}b = \pm ca \pm db$. Similarly to~\cite[Theorem 6.14]{brown}, we have $(ca)b = \pm c(ab)$. Then $0 = (\pm ca \pm db)\bar{b} = \pm (ca)\bar{b} \pm (db)\bar{b} = \pm (ca)b \pm n(b) d = \pm c(ab) \pm d$. Hence $0 = \bar{c}(\pm c(ab) \pm d) = \pm n(c)(ab) \pm \bar{c}d = \pm ab \pm cd$, so $ab = \pm cd$.
\end{proof}

\begin{corollary} \label{corollary:component-product-path}
Let $(a,b), (c,d) \in Z_e(\A_{n+1})$ belong to the same connected component of $\Gamma_e(\A_{n+1})$. Then $ab = \pm cd$.
\end{corollary}

\begin{proof}
Let $P$ be a path of length $k$ in $\Gamma_e(\A_{n+1})$ from $(a,b)$ to $(c,d)$. Then the statement follows immediately from Lemma~\ref{lemma:basis-pairs-zero-divisors-relation} by induction on $k$.
\end{proof}

\begin{notation}
We denote by $C(e_j^{(n)})$ the subgraph of $\Gamma_e(\A_{n+1})$ on the set of vertices $(a,b)$ such that $ab = \pm e_j^{(n)}$.
\end{notation}

It follows from Corollary~\ref{corollary:component-product-path} that $C(e_j^{(n)})$ and $C(e_k^{(n)})$, $j \neq k$, are not connected to each other in $\Gamma_e(\A_{n+1})$.

By~\cite[Lemma~4]{schafer}, $e^{(n)}_j$ is always alternative in $\A_n$, so $Z_e(\A_{n+1}) \subseteq DA(\A_{n+1})$. Moreover, for $(a,b) \in Z_e(\A_{n+1})$ we always have $n(a) = \pm 1$ and $n(b) = \pm 1$, and thus $(a,b)$ satisfies condition~\eqref{equation:norm-condition} for some $\chi = \pm 1$, see Remark~\ref{remark:norm-condition}. The values of $\chi$ can be different for different elements of $Z_e(\A_{n+1})$, so it becomes a function of a zero divisor: $\chi = \chi((a,b))$. For simplicity, we will omit the argument of $\chi$ if it is clear of the context.

As follows from Remark~\ref{remark:condition-asterisk}, one may prove by induction on the length of path that $\chi$ is constant on any connected component of $\Gamma_e(\A_{n+1})$. The value $\chi = \chi(C)$ is called the {\em characteristic} of the connected component $C \subseteq \Gamma_e(\A_{n+1})$. Below we determine the value of $\chi$ for every vertex of $C(e^{(n)}_j)$ even in the case when it is not connected.

\begin{proposition} \label{proposition:chi-value}
For any $x \in \E_n$ we have $\chi(C(x)) = \gamma_n n(x)$.
\end{proposition}

\begin{proof}
Let $(a,b) \in V(C(x))$, that is, $ab = \pm x$. Since $n(b) = \pm 1$ and both $a$ and $b$ are alternative, we obtain by Lemma~\ref{lemma:alternative-elements-are-normed} that
\[
\chi = \gamma_n \frac{n(a)}{n(b)} = \gamma_n n(a)n(b) = \gamma_n n(ab) = \gamma_n n(\pm x) = \gamma_n n(x). \tag*{\qedhere}
\]
\end{proof}

\begin{remark} \label{remark:component-types}
We consider four groups of elements $(x,y) \in Z_e'(\A_{n+1})$, depending on the type of $xy$. In this classification we assume that $e_0,a,b,c$ are linearly independent elements in $\pm \E_{n-1}$, and $ab = \pm c$.
\begin{enumerate}[label=(\Roman*)]
    \item If $xy \in \Lin(e_0)$, then $(x,y)$ is of the form $(a, \pm a)$, $(\til{a}, \pm \til{a})$, or $(\til{e}_0, \pm \til{e}_0)$.\label{item:component-type-1}
    \item If $xy \in \Lin(\til{e}_0)$, then $(x,y)$ is of the form $(a, \pm \til{a})$, $(\til{a}, \pm a)$, or $(\til{e}_0, \pm e_0)$.\label{item:component-type-2}
    \item If $xy \in \Lin(c)$, then $(x,y)$ is of the form $(c, e_0)$, $(\til{c}, \til{e}_0)$, $(\til{e}_0, \til{c})$, $(a, b)$, or $(\til{a}, \til{b})$.\label{item:component-type-3}
    \item If $xy \in \Lin(\til{c})$, then $(x,y)$ is of the form $(\til{c}, e_0)$, $(c, \til{e}_0)$, $(\til{e}_0, c)$, $(a, \til{b})$, or $(\til{a}, b)$.\label{item:component-type-4}
\end{enumerate}
By Corollary~\ref{corollary:component-product-path}, elements of distinct types (or of the same type but with linearly independent values of $c$) belong to distinct connected components of $\Gamma_e(\A_{n+1})$.
\end{remark}

\subsection{Establishing orthogonality conditions}

We now determine the conditions under which two elements of $Z_e'(\A_{n+1})$ are orthogonal.

Two following important propositions establish the ``almost equivalence'' of $(a,b)$, $(b, \gamma_n a)$, $(\til{a}, \til{b})$ and $(\til{b}, \gamma_n a)$ where $a,b,\til{a},\til{b}$ are linearly independent elements in $\pm \E''_n$. These elements are not properly equivalent in the sense that the sets of their neighbours do not coincide. However, all four elements have the same neighbours whose both components do not belong to $\mathbb{O}_{a,b}$, see Lemma~\ref{lemma:distinct-nonspecial-element}.

\begin{proposition} \label{proposition:shift-equivalence}
Let $a,b,c,d \in \A_n$ be pure. Then $(a,b)$ is orthogonal to $(c,d)$ if and only if it is orthogonal to $(d, \gamma_n c)$.
\end{proposition}

\begin{proof}
Since $\gamma_n^2 = 1$, we have $(a,b)(c,d) = (ac + \gamma_n \bar{d}b, da + b\bar{c}) = (ac - \gamma_n db, da - bc)$
\\
and $(a,b)(d,\gamma_n c) = (ad + \gamma_n \overline{(\gamma_n c)}b, \gamma_n ca + b\bar{d}) = (ad - cb, \gamma_n(ca - \gamma_n bd)) =
\\
(\overline{(da - bc)}, \gamma_n \overline{(ac - \gamma_n db)})$. Then $(a,b)(c,d) = 0$ if and only if $(a,b)(d,\gamma_n c) = 0$.
\end{proof}

\begin{proposition} \label{proposition:tilde-equivalence}
Let $a,\til{a},b,\til{b},c,\til{c},d,\til{d}$ be linearly independent elements in $\pm \E''_{n}$. Then $(a,b)$ is orthogonal to $(c,d)$ if and only if it is orthogonal to $(\til{c},\til{d})$.

Consequently, $(a,b)$ is orthogonal to $(c,d)$ if and only if $(\til{a},\til{b})$ is orthogonal to $(\til{c},\til{d})$.
\end{proposition}

\begin{proof}
By Lemma~\ref{lemma:octonionic-subalgebra}, any two elements of $a,b,c,d$ together with $\til{e}_0$ generate an octonionic subalgebra in~$\A_n$. It remains to see that $(a,b)(c,d) = (ac + \gamma_n \bar{d}b, da + b\bar{c}) = (ac - \gamma_n db, da - bc)$ and $(a,b)(\til{c},\til{d}) = (a\til{c} + \gamma_n \bar{\til{d}}\,b, \til{d}a + b\bar{\til{c}}) = (a\til{c} - \gamma_n \til{d}b, \til{d}a - b\til{c}) = (-\til{ac} + \gamma_n \til{da}, -\til{da} + \til{bc}) = (-\til{(ac - \gamma_n db)}, -\til{(da - bc)})$. Clearly, $(a,b)(c,d) = 0$ if and only if $(a,b)(\til{c},\til{d}) = 0$.

For the second part of the statement, we apply the proposition twice. First we say that $(a,b)$ is orthogonal to $(c,d)$ if and only $(a,b)$ it is orthogonal to $(\til{c},\til{d})$, and then that $(\til{c},\til{d})$ is orthogonal to $(a,b)$ if and only if $(\til{c},\til{d})$ is orthogonal to $(\til{a},\til{b})$.
\end{proof}

We first consider elements of type~\eqref{item:component-type-1} from Remark~\ref{remark:component-types}.

\begin{lemma} \label{lemma:type-1-element}
Consider $(a, \pm a) \in \A_{n+1}$ with $a \in \E'_n$. Note that $\chi = \gamma_n \frac{n(a)}{n(\pm a)} = \gamma_n$. Then $(a, \pm a) \in Z'_e(\A_{n+1})$ if and only if $\chi = \gamma_n = 1$, and in this case it is orthogonal to itself and to all $(b, \mp b)$ such that $b \in \E'_n \setminus \{ a \}$.
\end{lemma}

\begin{proof}
By Lemma~\ref{lemma:split-algebras-orthogonalizers},
$$
O_{\A_{n+1}}((a, \pm a)) = \left\{ \left(b, \mp \dfrac{(ab)a}{n(a)} \right) \; \bigg| \; \Re(b) = 0, \: a(ba) = \chi (ab)a \right\}.
$$
The explicit form of annihilators is described by Lemma~\ref{lemma:split-algebras-annihilators}. By flexibility, we always have $a(ba) = (ab)a$, and $a(ba) \neq 0$ for $b \neq 0$. Therefore, for $\chi = -1$ we have $l.\Ann_{\A_{n+1}}((a, \pm a)) = r.\Ann_{\A_{n+1}}((a, \pm a)) = O_{\A_{n+1}}((a, \pm a)) = \{ 0 \}$, and thus $(a, \pm a) \notin Z'_e(\A_{n+1})$. Assume now that $\chi = 1$. By Proposition~\ref{proposition:self-orthogonality}, $(a, \pm a)$ is orthogonal to itself. Consider now $b \in \E'_n$, $b \neq a$. Then $b \perp a$, and thus $ab + ba = 0$. Therefore,
$$
\mp \dfrac{(ab)a}{n(a)} = \pm \dfrac{(ba)a}{n(a)} = \mp \dfrac{b(a\bar{a})}{n(a)} = \mp b,
$$
so $(a, \pm a)$ is orthogonal to $(b, \mp b)$.
\end{proof}

We proceed to doubly pure elements of type~\eqref{item:component-type-2}. Those elements of type~\eqref{item:component-type-2} which have $e_0$ as the second component are described in Lemma~\ref{lemma:e_0-special-element}.

\begin{lemma} \label{lemma:type-2-element}
Consider $(\til{a}, \pm a) \in \A_{n+1}$ with $a \in \E''_n = \E_n \setminus \{ e_0, \til{e}_0 \}$. Note that $\chi = \gamma_n \frac{n(\til{a})}{n(\pm a)} = \gamma_n \frac{n(\til{e}_0)n(a)}{n(a)} = - \gamma_{n-1} \gamma_n$.
\begin{enumerate}
    \item If $\chi = 1$, then $(\til{a}, \pm a)$ is orthogonal to itself, $(\til{e}_0, \pm e_0)$, and $(a, \pm \gamma_n \til{a})$.
    \item If $\chi = -1$, then $(\til{a}, \pm a)$ is connected in $\Gamma_e(\A_{n+1})$ to all $(\til{b}, \mp b)$ such that $b \in \E''_n \setminus \{ a, \til{a} \}$, and only to them. Besides, in this case $(\til{a}, \pm a) \in Z'_e(\A_{n+1})$ if and only if $n \geq 3$.
\end{enumerate}
\end{lemma}

\begin{proof}
By Lemma~\ref{lemma:split-algebras-orthogonalizers},
$$
O_{\A_{n+1}}((\til{a}, \pm a)) = \left\{ \left(b, \mp \dfrac{(ab)\til{a}}{n(\til{a})} \right) \; \bigg| \; \Re(b) = 0, \: a(b\til{a}) = \chi (ab)\til{a} \right\}.
$$
Since $a \in \pm \E''_n$, we may apply Lemma~\ref{lemma:quaternionic-subalgebra} to it. Thus for $b \in \mathbb{H}_a = \Lin(1,a,\til{e}_0,\til{a})$ we have $a(b\til{a}) = (ab)\til{a}$. On the other hand, if $b \perp \mathbb{H}_a$, then we may apply Lemma~\ref{lemma:octonionic-subalgebra} to $a, b$ and obtain that $a(b\til{a}) = -(ab)\til{a}$. It remains to note that for any $b \in \E'_n$ we have either $b \in \mathbb{H}_a$ or $b \perp \mathbb{H}_a$.
\begin{enumerate}
    \item Assume that $\chi = 1$. By Proposition~\ref{proposition:self-orthogonality}, $(\til{a}, \pm a)$ is orthogonal to itself. Then Proposition~\ref{proposition:shift-equivalence} implies that $(\til{a}, \pm a)$ is also orthogonal to $(\pm a, \gamma_n \til{a})$. Since $b \in \mathbb{H}_a$ and $b \in \pm \E'_n$, it remains to consider $b = \til{e}_0$. Clearly,
    $$
    \mp \dfrac{(a\til{e}_0)\til{a}}{n(\til{a})} = \mp \dfrac{\til{a}\til{a}}{n(\til{a})} = \pm e_0.
    $$
    Hence $(\til{a}, \pm a)$ is orthogonal to $(\til{e}_0, \pm e_0)$.
    \item If $\chi = -1$, then for $b \in \E'_n$ we have $b \perp \mathbb{H}_a$, so we may replace $b$ with $\til{b}$ for some $b \in \E''_n$ and $b \neq a$, $b \neq \til{a}$. We use Lemma~\ref{lemma:octonionic-subalgebra} to compute
    $$
    \mp \dfrac{(a\til{b})\til{a}}{n(\til{a})} = \pm \dfrac{(\til{ab})\til{a}}{n(\til{a})} = \mp \dfrac{n(a)n(\til{e}_0)b}{n(a)n(\til{e}_0)} = \mp b.
    $$
    Hence $(\til{a}, \pm a)$ is orthogonal to $(\til{b}, \mp b)$.
    
    Note that if $n \leq 2$, then $\A_n = \mathbb{H}_a$, so $a(b\til{a}) = (ab)\til{a}$ for all $b \in \A_n$, and thus it follows from Lemma~\ref{lemma:split-algebras-annihilators} that  $l.\Ann_{\A_{n+1}}((\til{a}, \pm a)) = r.\Ann_{\A_{n+1}}((\til{a}, \pm a)) = O_{\A_{n+1}}((\til{a}, \pm a)) = \{ 0 \}$. However, if $n \geq 3$, then $O_{\A_{n+1}}((\til{a}, \pm a)) \neq \{ 0 \}$ and $(\til{a}, \pm a) \in Z'_e(\A_{n+1})$. \qedhere
\end{enumerate}
\end{proof}

Let us now explore elements of types~\eqref{item:component-type-3} and~\eqref{item:component-type-4} from Remark~\ref{remark:component-types}. We begin with those of them which are not four times pure, see Definition~\ref{definition:real-imaginary-part}. That is, one of their components is equal either to $\pm e_0$ or to $\pm \til{e}_0$. Clearly, $(e_0, c) \notin Z'_e(\A_{n+1})$ for any $c \in \pm \E_n$, so we will consider only elements of the form $(c, e_0)$.

\begin{lemma} \label{lemma:e_0-special-element}
Consider $(c, e_0) \in \A_{n+1}$, $c \in \pm \E'_n$. Then $(c, e_0) \in Z'_e(\A_{n+1})$ if and only if its $\chi = 1$, and in this case $(c, e_0)$ is orthogonal to $(a,b) \in Z'_e(\A_{n+1})$ if and only if $ab = n(b)c$.
\end{lemma}

\begin{proof}
By Lemma~\ref{lemma:split-algebras-orthogonalizers},
$$
O_{\A_{n+1}}((c, e_0)) = \left\{ \left(a, -\dfrac{(e_0a)c}{n(c)} \right) \; \bigg| \; \Re(a) = 0, \: e_0(ac) = \chi (e_0a)c \right\}.
$$
The explicit form of annihilators is described in Lemma~\ref{lemma:split-algebras-annihilators}. Clearly, we always have $e_0(ac) = ac = (e_0a)c$, and $ac \neq 0$ for $a \neq 0$. Therefore, for $\chi = -1$ we have $l.\Ann_{\A_{n+1}}((c, e_0)) = r.\Ann_{\A_{n+1}}((c, e_0)) = O_{\A_{n+1}}((c, e_0)) = \{ 0 \}$, and thus $(c, e_0) \notin Z'_e(\A_{n+1})$. Assume now that $\chi = 1$. Let $a \in \E'_n$. We denote $b = -\dfrac{ac}{n(c)} \in \pm \E_n$. Since $a$ is alternative, by Lemma~\ref{lemma:alternative-elements-are-normed},
$$
n(b) = \frac{n(ac)}{(n(c))^2} = \frac{n(a)n(c)}{(n(c))^2} = \frac{n(a)}{n(c)},
$$
and thus
$$
ab = -\dfrac{a(ac)}{n(c)} = \dfrac{(a\bar{a})c}{n(c)} = \dfrac{n(a)}{n(c)}c = n(b)c.
$$
Conversely, if $ab = n(b)c$, then $b = -\dfrac{ac}{n(c)} \in \pm \E_n$.
\end{proof}

We now investigate the elements with one of the components being equal to $\til{e}_0$.

\begin{lemma} \label{lemma:til-e_0-special-element}
Consider $(c, \til{e}_0), (\til{e}_0, \gamma_n c) \in \A_{n+1}$, $c \in \pm \E''_n$. Clearly, their $\chi$ are equal.
\begin{enumerate}
    \item If $\chi = 1$, then $(c, \til{e}_0)$ and $(\til{e}_0, \gamma_n c)$ are orthogonal to themselves, and all edges in $\Gamma_e(\A_{n+1})$ containing them can be depicted as
    $$
    (\til{c}, -\gamma_{n-1} e_0) \longleftrightarrow (c, \til{e}_0) \longleftrightarrow (\til{e}_0, \gamma_n c) \longleftrightarrow (\til{c}, \gamma_{n-1} e_0).
    $$
    \item Let now $\chi = -1$. Then $(c, \til{e}_0)$ and $(\til{e}_0, \gamma_n c)$ are connected in $\Gamma_e(\A_{n+1})$ to all $(a,b) \in Z'_e(\A_{n+1})$ such that $a \perp \mathbb{H}_c = \Lin(e_0, c, \til{e}_0, \til{c})$ and $ab = \gamma_{n-1}n(b) \til{c}$, and only to them. Besides, in this case $(c, \til{e}_0), (\til{e}_0, \gamma_n c) \in Z'_e(\A_{n+1})$ if and only if $n \geq 3$.
\end{enumerate}
\end{lemma}

\begin{proof}
By Lemma~\ref{lemma:split-algebras-orthogonalizers},
$$
O_{\A_{n+1}}((c, \til{e}_0)) = \left\{ \left(a, -\dfrac{(\til{e}_0a)c}{n(c)} \right) \; \bigg| \; \Re(a) = 0, \: \til{e}_0(ac) = \chi (\til{e}_0a)c \right\}.
$$
Since $c \in \pm \E''_n$, we may apply Lemma~\ref{lemma:quaternionic-subalgebra} to it. Thus for $a \in \mathbb{H}_c$ we have $\til{e}_0(ac) = (\til{e}_0a)c$. On the other hand, if $a \perp \mathbb{H}_c$, then we may apply Lemma~\ref{lemma:octonionic-subalgebra} to $a, c$ and obtain that $\til{e}_0(ac) = -(\til{e}_0a)c$. It remains to note that for any $a \in \E'_n$ we have either $a \in \mathbb{H}_c$ or $a \perp \mathbb{H}_c$.
\begin{enumerate}
    \item Assume that $\chi = 1$. By Proposition~\ref{proposition:self-orthogonality}, $(c, \til{e}_0)$ and $(\til{e}_0, \gamma_n c)$ are orthogonal to themselves. Hence it follows from Proposition~\ref{proposition:shift-equivalence} that $(c, \til{e}_0)$ is also orthogonal to $(\til{e}_0, \gamma_n c)$. Since $a \in \mathbb{H}_c$ and $a \in \pm \E'_n$, it remains to consider $a = \til{c}$. It follows from Lemma~\ref{lemma:quaternionic-subalgebra} that
    $$
    -\dfrac{(\til{e}_0\til{c})c}{n(c)} = -\dfrac{-\gamma_{n-1} cc}{n(c)} = -\dfrac{\gamma_{n-1} c\bar{c}}{n(c)} = -\gamma_{n-1} e_0.
    $$
    We obtain that $(c, \til{e}_0)$ is orthogonal to $(\til{c}, -\gamma_{n-1} e_0)$. Similarly, $(\til{e}_0, \gamma_n c)$ is orthogonal to $(\til{c}, \gamma_{n-1} e_0)$.
    \item If $\chi = -1$, then it is clear from Lemma~\ref{lemma:split-algebras-orthogonalizers} that any neighbour $(a, b)$ either of $(c, \til{e}_0)$ or of $(\til{e}_0, \gamma_n c)$ in $\Gamma_e(\A_{n+1})$ is doubly pure. It follows from Proposition~\ref{proposition:shift-equivalence} that $(a,b)$ is orthogonal to $(c, \til{e}_0)$ if and only if it is orthogonal to $(\til{e}_0, \gamma_n c)$. Consider now some $a \in \E'_n$ such that $a \perp \mathbb{H}_c$, and let $b = -\dfrac{(\til{e}_0a)c}{n(c)} \in \E_n$. By Lemma~\ref{lemma:alternative-elements-are-normed},
    $$
    n(b) = \frac{n((\til{e}_0a)c)}{(n(c))^2} = \frac{n(\til{e}_0a)n(c)}{(n(c))^2} = \frac{n(\til{e}_0)n(a)}{n(c)} = -\gamma_{n-1} \frac{n(a)}{n(c)},
    $$
    We use Lemma~\ref{lemma:octonionic-subalgebra} to compute
    $$
    ab = a \cdot \left(-\dfrac{(\til{e}_0a)c}{n(c)}\right) = a\dfrac{\til{a}c}{n(c)} = -a\dfrac{\til{ac}}{n(c)} = -\dfrac{n(a)\til{c}}{n(c)} = \gamma_{n-1}n(b) \til{c}.
    $$
    Conversely, if $ab = \gamma_{n-1}n(b) \til{c}$, then $b = -\dfrac{(\til{e}_0a)c}{n(c)} \in \pm \E_n$, so $(a,b)$ is orthogonal to $(c, \til{e}_0)$.
    
    Note that if $n \leq 2$, then $\A_n = \mathbb{H}_c$, so $\til{e}_0(ac) = (\til{e}_0a)c$ for all $a \in \A_n$, and thus it follows from Lemma~\ref{lemma:split-algebras-annihilators} that  $l.\Ann_{\A_{n+1}}((c, \til{e}_0)) = r.\Ann_{\A_{n+1}}((c, \til{e}_0)) = l.\Ann_{\A_{n+1}}((\til{e}_0, \gamma_n c)) = r.\Ann_{\A_{n+1}}((\til{e}_0, \gamma_n c)) = \{ 0 \}$. However, if $n \geq 3$, then
    \\
    $O_{\A_{n+1}}((c, \til{e}_0)) = O_{\A_{n+1}}((\til{e}_0, \gamma_n c)) \neq \{ 0 \}$ and $(c, \til{e}_0), (\til{e}_0, \gamma_n c) \in Z'_e(\A_{n+1})$. \qedhere
\end{enumerate}
\end{proof}

Further we study orthogonality relation of four times pure elements.

\begin{notation}
We denote
\begin{align*}
    \A_{n}^{\circ} &= \A_{n} \{\gamma_0, \dots, \gamma_{n-2}, \gamma_n \},\\
    \A_{n}^{\bullet} &= \A_{n} \{\gamma_0, \dots, \gamma_{n-2}, \gamma_{n-1} \gamma_n \}.
\end{align*}
\end{notation}

\begin{lemma} \label{lemma:distinct-nonspecial-element}
Let $a,b,c,d$ be linearly independent elements in $\pm \E'_{n-1}$.
\begin{enumerate}
    \item Let $x \in \{ (a, b), (b, \gamma_n a), (\til{a}, \til{b}), (\til{b}, \gamma_n \til{a}) \}$, $y \in \{ (c, d), (c, \gamma_n d), (\til{c}, \til{d}), (\til{c}, \gamma_n \til{d}) \}$. Then $x$ and $y$ are orthogonal in $\A_{n+1}$ if and only if $(a,b)$ and $(c,d)$ are orthogonal in~$\A_{n}^{\circ}$.
    \item Let $x \in \{ (a, \til{b}), (\til{b}, \gamma_n a), (\til{a}, \gamma_{n-1} b), (b, \gamma_{n-1} \gamma_n \til{a}) \}$, $y \in \{ (c, \til{d}), (\til{d}, \gamma_n c),
    \\
    (\til{c}, \gamma_{n-1} d), (d, \gamma_{n-1} \gamma_n \til{c}) \}$. Then $x$ and $y$ are orthogonal in $\A_{n+1}$ if and only if $(a,b)$ and $(c,d)$ are orthogonal in~$\A_{n}^{\bullet}$.
\end{enumerate}
\end{lemma}

\begin{proof}
\leavevmode
\begin{enumerate}
    \item Clearly, $(a,b)$ and $(c,d)$ are orthogonal in $\A_{n+1}$ if and only if $(ac + \gamma_n \bar{d}b, da + b\bar{c}) = 0$. This condition holds exactly when $(a,b)$ and $(c,d)$ are orthogonal in $\A_{n}^{\circ}$. The equivalence of other orthogonalities follows immediately from Propositions~\ref{proposition:shift-equivalence} and~\ref{proposition:tilde-equivalence}.
    \item By Lemma~\ref{lemma:tilde-properties}, $\til{\til{b}} = \gamma_{n-1} b$ and $\gamma_{n-1}^2 = 1$, so $(\til{a},\til{\til{b}}) = (\til{a}, \gamma_{n-1} b)$ and $(\til{\til{b}}, \til{\gamma_n a}) = (\gamma_{n-1} b, \gamma_n \til{a}) = \gamma_{n-1} (b, \gamma_{n-1} \gamma_n \til{a})$. Similarly, $(\til{c},\til{\til{d}}) = (\til{c}, \gamma_{n-1} d)$ and $(\til{\til{d}}, \til{\gamma_n c}) = \gamma_{n-1} (d, \gamma_{n-1} \gamma_n \til{c})$. Then, by Propositions~\ref{proposition:shift-equivalence} and~\ref{proposition:tilde-equivalence}, $x \in \{ (a, \til{b}), (\til{b}, \gamma_n a), (\til{a}, \gamma_{n-1} b),\\ (b, \gamma_{n-1} \gamma_n \til{a}) \}$ and $y \in \{ (c, \til{d}), (\til{d}, \gamma_n c), (\til{c}, \gamma_{n-1} d),
    \\
    (d, \gamma_{n-1} \gamma_n \til{c}) \}$ are orthogonal in $\A_{n+1}$ if and only if $(a,\til{b})$ and $(c,\til{d})$ are orthogonal in $\A_{n+1}$.
    
    By Lemma~\ref{lemma:octonionic-subalgebra}, any two elements of $a,b,c,d$ together with $\til{e}_0$ generate an octonionic subalgebra in~$\A_n$. Hence $(a,\til{b})$ and $(c,\til{d})$ are orthogonal in $\A_{n+1}$ if and only if 
    $$
    0 = (ac + \gamma_n \bar{\til{d}}\,\til{b}, \til{d}a + \til{b}\bar{c}) = (ac + \gamma_{n-1} \gamma_n \bar{d}b, -\til{(da + b\bar{c})}).
    $$
    Since in $\A^{\bullet}_n$ we have $(a,b)(c,d) = (ac + \gamma_{n-1} \gamma_n \bar{d}b, da + b\bar{c})$, this condition holds exactly when $(a,b)$ and $(c,d)$ are orthogonal in $\A_{n}^{\bullet}$. \qedhere
\end{enumerate}
\end{proof}

We now consider the case when a four times pure element $(a,b)$ is orthogonal to some element such that one of its components equals either $\pm a$ or $\pm \til{a}$ (hence the other component is either $\pm b$ or $\pm \til{b}$).

\begin{lemma} \label{lemma:common-nonspecial-element}
Let $a,b,\til{a},\til{b}$ be linearly independent elements in $\pm \E''_n$.
\begin{enumerate}
    \item If $\chi = 1$, then $(a,b)$ is orthogonal to itself and $(b, \gamma_n a)$.
    \item If $\chi = -1$, then $(a,b)$ is orthogonal to $(\til{a}, \til{b})$ and $(\til{b}, \gamma_n \til{a})$.
\end{enumerate}
\end{lemma}

\begin{proof}
By Lemma~\ref{lemma:split-algebras-orthogonalizers},
$$
O_{\A_{n+1}}((a,b)) = \left\{ \left(c, -\dfrac{(bc)a}{n(a)} \right) \; \bigg| \; \Re(c) = 0, \: b(ca) = \chi (bc)a \right\}.
$$
It follows from Lemma~\ref{lemma:octonionic-subalgebra} that $\til{e}_0, a, b$ form an octonionic subalgebra in $\A_{n+1}$. Hence for $c \in \{ \pm a, \pm b \}$ we have $b(ca) = (bc)a$, and for $c \in \{ \pm \til{a}, \pm \til{b} \}$ we have $b(ca) = -(bc)a$. Note that $b(ca) \neq 0$ in all cases. Therefore, for $\chi = 1$ we take $c \in \{ a, b \}$, and for $\chi = -1$ we take $c \in \{ \til{a}, \til{b} \}$.
\begin{enumerate}
    \item If $\chi = 1$, then, by Proposition~\ref{proposition:self-orthogonality}, $(a,b)$ is orthogonal to itself. Hence, by Proposition~\ref{proposition:shift-equivalence}, $(a,b)$ is also orthogonal to $(b, \gamma_n a)$. Note that we use only the fact that $a,b \in \E'_n$.
    \item If $\chi = -1$, then we consider first $c = \til{a}$. By Lemma~\ref{lemma:octonionic-subalgebra},
    $$
    -\dfrac{(b\til{a})a}{n(a)} = -\dfrac{(\til{ab})a}{n(a)} = -\dfrac{-n(a) \til{b}}{n(a)} = \til{b},
    $$
    so $(a,b)$ is orthogonal to $(\til{a}, \til{b})$. Then, by Proposition~\ref{proposition:shift-equivalence}, $(a,b)$ is also orthogonal to $(\til{b}, \gamma_n \til{a})$. \qedhere
\end{enumerate}
\end{proof}

\begin{theorem} \label{theorem:zero-divisors-criterion}
Let $n \geq 1$, $(a,b) \in \E'_n \times (\pm \E_n)$. Then $(a,b) \in Z(\A_{n+1})$, except for the following cases:
\begin{enumerate}
    \item $n \leq 2$ and $\chi = -1$;
    \item $b = \pm a$ and $\chi = \gamma_n = -1$;
    \item $b = \pm e_0$ and $\chi = \gamma_n n(a) = -1$.
\end{enumerate}
\end{theorem}

\begin{proof}
If $\chi = 1$, then $(a,b)$ is orthogonal to itself by Proposition~\ref{proposition:self-orthogonality}, and thus $(a,b)$ is a zero divisor. Hence it is sufficient to consider the case when $\chi = -1$:
\begin{itemize}
\item If $a,b,\til{a},\til{b}$ are linearly independent elements in $\pm \E''_n$, then $(a,b)$ is a zero divisor by Lemma~\ref{lemma:common-nonspecial-element}. This case is only possible for $n \geq 3$.
\item If $a \in \E'_n$ and $b = \pm a$, then, by Lemma~\ref{lemma:type-1-element}, $(a,b)$ is not a zero divisor. This includes the case when $a = \til{e}_0$ and $b = \pm \til{e}_0$.
\item If $a \in \E''_n$ and $b = \pm \til{a}$, then, by Lemma~\ref{lemma:type-2-element}, $(a,b)$ is a zero divisor if and only if $n \geq 3$.
\item If $b = \pm e_0$, then, by Lemma~\ref{lemma:e_0-special-element}, $(a,b)$ is not a zero divisor.
\item If $a = \til{e}_0$ and $b \in \pm E''_n$, then, by Lemma~\ref{lemma:til-e_0-special-element}, $(a,b)$ is a zero divisor if and only if $n \geq 3$. \qedhere
\end{itemize}
\end{proof}

\medskip

The author is grateful to her scientific advisor Professor Alexander E. Guterman for posing the
problem and fruitful discussions.

\end{document}